\documentclass[11pt,a4paper]{amsart}
\usepackage{graphics,epic}
\usepackage{amsmath,amssymb, amsthm}
\usepackage[all,2cell]{xy}
\usepackage{xcolor}

\newtheorem{theorem}{Theorem}[section]
\newtheorem*{theorem*}{Theorem}

\newtheorem{lemma}[theorem]{Lemma}
\newtheorem{proposition}[theorem]{Proposition}
\newtheorem{corollary}[theorem]{Corollary}

\newtheorem*{conjecture*}{Conjecture}

\newtheorem{remark}[theorem]{Remark}
\newtheorem{definition}[theorem]{Definition}

\newcommand{\ie}{{\em i.e.}\ }
\newcommand{\confer}{{\em cf.}\ }

\newcommand{\opname}[1]{\operatorname{\mathsf{#1}}}

\renewcommand{\mod}{\opname{mod}\nolimits}

\newcommand{\proj}{\opname{proj}\nolimits}
\newcommand{\inj}{\opname{inj}\nolimits}
\newcommand{\Mod}{\opname{Mod}\nolimits}

\newcommand{\per}{\opname{per}\nolimits}
\newcommand{\add}{\opname{add}\nolimits}
\newcommand{\Add}{\opname{Add}\nolimits}

\newcommand{\thick}{\opname{thick}\nolimits}

\newcommand{\nilrep}{\opname{nil.rep}\nolimits}

\newcommand{\cok}{\opname{cok}\nolimits}

\renewcommand{\ker}{\opname{ker}\nolimits}

\newcommand{\Hom}{\opname{Hom}}

\newcommand{\cHom}{\mathcal{H}\it{om}}

\newcommand{\Ext}{\opname{Ext}}

\newcommand{\Aut}{\opname{Aut}}
\newcommand{\End}{\opname{End}}

\newcommand{\ten}{\otimes}
\newcommand{\lten}{\overset{\boldmath{L}}{\ten}}

\newcommand{\ca}{{\mathcal A}}

\newcommand{\cc}{{\mathcal C}}
\newcommand{\cd}{{\mathcal D}}

\newcommand{\cf}{{\mathcal F}}

\newcommand{\ck}{{\mathcal K}}

\newcommand{\cm}{{\mathcal M}}

\newcommand{\co}{{\mathcal O}}
\newcommand{\cp}{{\mathcal P}}

\newcommand{\cs}{{\mathcal S}}
\newcommand{\ct}{{\mathcal T}}

\newcommand{\cx}{{\mathcal X}}

\numberwithin{equation}{section}

\begin{document}

\title[Silting objects, simple-minded collections, $t$-structures and
co-$t$-structures]{Silting objects, simple-minded collections,
$t$-structures and co-$t$-structures for finite-dimensional algebras}

\author{Steffen Koenig}
\address{Steffen Koenig, Universit\"{a}t Stuttgart\\
Institut f\"ur Algebra und Zahlentheorie\\ Pfaffenwaldring 57\\ D-70569 Stuttgart\\
Germany} \email{skoenig@mathematik.uni-stuttgart.de}

\author{Dong Yang}
\address{Dong Yang, Department of Mathematics, Nanjing University, Nanjing 210093, P. R. China}
\email{dongyang2002@gmail.com}
\date{Last modified on \today}%October 2008.

%\begin{abstract} It is shown that in the derived category of a finite-dimensional algebra
%there are one-to-one correspondences between silting objects,
%simple-minded collections, bounded $t$-structures with length heart
%and bounded co-$t$-structures. Moreover, these correspondences
%commute with mutations. Our approach is `dual' (parallel? similar?) to that of
%Keller--Nicol\'as and that of Vit\'oria.
%\end{abstract}

\begin{abstract}
Bijective correspondences are established between (1) silting
objects, (2) simple-minded collections, (3) bounded $t$-structures
with length heart and (4) bounded co-$t$-structures. These
correspondences are shown to commute with mutations and partial orders. The results are
valid for finite dimensional algebras. {A concrete example is given to illustrate how these correspondences help to compute the space of Bridgeland's stability conditions.}
\end{abstract}

\maketitle

\tableofcontents

\section{Introduction}

Let $\Lambda$ be a finite-dimensional associative algebra. Fundamental objects of study in the representation theory
of $\Lambda$ are the projective modules, the simple modules and the
category of all (finite-dimensional) $\Lambda$-modules. Various
structural concepts have been introduced that include one of these
classes of objects as particular instances. In this article, four
such concepts are related by explicit bijections. Moreover, these
bijections are shown to commute with the basic operation of mutation
and to preserve partial orders.

These four concepts may be based on two different general points of view,
either considering particular generators of categories ((1) and (2)) or
considering structures on categories that identify particular subcategories
((3) and (4)):
\begin{itemize}
\item[(1)]
Focussing on objects that generate categories, the theory of Morita
equivalences has been extended to tilting or derived equivalences.
In this way, projective generators are examples of tilting modules,
which have been generalised further to {\it silting objects} (which
are allowed to have negative self-extensions).
\item[(2)]
Another, and different, natural choice of `generators' of a module category
is the set of simple modules (up to isomorphism). In the context of
derived or stable equivalences, this set is included in the concept
of simple-minded system or {\it simple-minded collection}.
\item[(3)]
Starting with a triangulated category and looking for particular
subcategories, {\it $t$-structures} have been defined so as to provide abelian
categories as their hearts. The finite-dimensional $\Lambda$-modules form
the heart of some $t$-structure in the bounded derived category
$D^b(\mod \Lambda)$.
\item[(4)]
Choosing as triangulated category the homotopy category $K^b(\proj \Lambda)$,
one considers {\it co-$t$-structures}. The additive category $\proj \Lambda$
occurs as the co-heart of some co-$t$-structure in $K^b(\proj \Lambda)$.
\end{itemize}

The first main result of this article is:
\begin{theorem*}[\bf \ref{maintheorembijections}]
Let $\Lambda$ be a finite-dimensional algebra over a field $K$.
There are one-to-one correspondences between
\begin{itemize}
 \item[(1)] equivalence classes of silting objects in $K^b(\proj \Lambda)$,
 \item[(2)] equivalence classes of simple-minded collections in $\cd^b(\mod \Lambda)$,
 \item[(3)] bounded $t$-structures of $\cd^b(\mod \Lambda)$ with length heart,
 \item[(4)] bounded co-$t$-structures of $K^b(\proj \Lambda)$.
\end{itemize}
\end{theorem*}
Here two sets of objects in a category are \emph{equivalent} if they
additively generate the same subcategory.

\medskip

A common feature of all four concepts it that they allow for comparisons,
often by equivalences. In particular,
each of the four structures to be related comes with a basic operation, called
mutation, which produces a new such structure from a given one. Moreover, on each of the four structures there is a partial order. All the
bijections in Theorem \ref{maintheorembijections}
enjoy the following naturality properties:
\begin{theorem*}[\bf \ref{maintheoremmutations}]
Each of the bijections between the four structures (1), (2), (3) and (4)
commutes with the respective operation of mutation.
\end{theorem*}

\begin{theorem*}[\bf \ref{t:main-thm-partially-ordered-sets}]
Each of the bijections between the four structures (1), (2), (3) and (4)
preserves the respective partial orders.
\end{theorem*}

{The four concepts are crucial in representation theory, geometry
and topology. They are also closely related to fundamental concepts 
in cluster theory such as clusters (\cite{FominZelevinsky02}),
c-matrices and g-matrices (\cite{FominZelevinsky07,NakanishiZelevinsky12}) and 
cluster-tilting objects (\cite{BuanMarshReinekeReitenTodorov06}). We refer to the survey paper~\cite{BruestleYang13} for more details.}
A concrete example to be given at the end of the article
demonstrates one practical use of these bijections and their properties.

Finally we give some remarks on the literature.
For path algebras of Dynkin quivers, Keller and Vossieck \cite{KellerVossieck88}
have already given a bijection between bounded $t$-structures and silting
objects. The bijection between silting objects and $t$-structures with length heart
has been established by Keller and Nicol\'as~\cite{KellerNicolas12} for
homologically smooth non-positive dg algebras,
 by Assem, Souto Salorio and Trepode~\cite{AssemSoutoTrepode08} and
by Vit\'oria \cite{Vitoria12}, who are focussing on piecewise hereditary
algebras. An unbounded version of
this bijection has been studied by Aihara and Iyama~\cite{AiharaIyama12}. The bijection between
simple-minded collections and bounded $t$-structures has been established implicitely in Al-Nofayee's work~\cite{Al-Nofayee09}
and explicitely for homologically smooth non-positive dg algebras in Keller and Nicol\'as' work~\cite{KellerNicolas12} 
 and for finite-dimensional algebras in our preprint
\cite{KoenigYang10}, which has been {partly}
incorporated into the present article{, and partially in the work~\cite{RickardRouquier10} of Rickard and Rouquier}.
For hereditary algebras, Buan, Reiten and Thomas~\cite{BuanReitenThomas12} studied the bijections between silting objects, 
simple-minded collections (=$\mathrm{Hom}_{\leq 0}$-configurations in their setting) and bounded $t$-structures.
The correspondence between silting objects
and co-$t$-structures appears implicitly on various levels of generality
in the work of Aihara and Iyama~\cite{AiharaIyama12}
and of Bondarko~\cite{Bondarko10} and explicitly in full generality in
the work of Mendoza, S\'aenz, Santiago
and Souto Salorio~\cite{MendozaSaenzSantiagoSouto10} and of
Keller and Nicol\'as~\cite{KellerNicolas11}.
For homologically smooth non-positive dg algebras, all the bijections are
due to Keller and Nicol\'as~\cite{KellerNicolas11}. 
The intersection of our results with
those of Keller and Nicol\'as is the case of finite-dimensional algebras of
finite global dimension. 

\medskip
\noindent{\it Acknowledgement.}
The authors would like to thank Paul Balmer, {Mark Blume}, Martin Kalck, {Henning Krause},
Qunhua Liu, Yuya Mizuno, David Pauksztello, Pierre-Guy Plamondon, {David Ploog}, Jorge  Vit\'{o}ria and Jie Xiao 
for inspiring discussions and helpful remarks. 
%{ They thank the referee for...}
The second-named author gratefully acknowledges financial support
from Max-Planck-Institut f\"ur Mathematik in Bonn and from DFG
program SPP 1388 (YA297/1-1). 
%He thanks Paul Balmer, Martin Kalck,
%Yuya Mizuno, Pierre-Guy Plamondon and Jie Xiao for helpful remarks
%and thanks Martin Kalck, Jorge Vit\'{o}ria and Qunhua Liu for
%inspiring discussions. 
He is deeply grateful to Bernhard Keller for
valuable conversations on derived categories of dg algebras.

\section{Notations and preliminaries} \label{section-prelim}

\subsection{Notations}
Throughout, $K$ will be a field. All algebras, modules, vector
spaces and categories are over the base field $K$, and
$\mathrm{D}=\Hom_K(?,K)$ denotes the $K$-dual. By abuse of
notation, we will denote by $\Sigma$ the suspension functors of all
the triangulated categories.

For a category $\cc$, we
denote by $\Hom_{\cc}(X,Y)$ the morphism space from $X$ to $Y$,
where $X$ and $Y$ are two objects of $\cc$. We will omit the
subscript and write $\Hom(X,Y)$ when it does not cause confusion.
For $\cs$ a set of objects or a subcategory of $\cc$, call
\[{}^{\perp}\cs=\{X\in\cc\mid\Hom(X,S)=0 \text{ for all }S\in\cs\}\]
and
\[\cs^\perp=\{X\in\cc\mid\Hom(S,X)=0 \text{ for all }S\in\cs\}\]
the \emph{left} and \emph{right perpendicular category} of $\cs$, respectively.

Let $\cc$ be an additive category and $\cs$ a set of objects or a subcategory of $\cc$.
Let $\Add(\cs)$ and $\add(\cs)$, respectively, denote the smallest
full subcategory of $\cc$ containing all objects of $\cs$ and stable
for taking direct summands and coproducts respectively taking finite coproducts.
The category $\add(\cs)$ will be called the \emph{additive closure} of $\cs$.
If further $\cc$ is abelian or triangulated, the \emph{extension closure} of $\cs$ is the smallest subcategory
of $\cc$ containing $\cs$ and stable under taking extensions.
Assume that $\cc$ is triangulated and let
$\thick(\cs)$ denote the smallest triangulated subcategory of $\cc$ containing
objects in $\cs$ and stable under taking direct summands. We say that $\cs$ is a \emph{set of generators} of $\cc$,
or that $\cc$ is \emph{generated by} $\cs$, when $\cc=\thick(\cs)$.

\subsection{Derived categories}
For a finite-dimensional algebra $\Lambda$, let $\Mod\Lambda$
(respectively, $\mod\Lambda$, $\proj\Lambda$, $\inj\Lambda$) denote the category of right
$\Lambda$-modules (respectively, finite-dimensional right
$\Lambda$-modules, finite-dimensional projective, injective right $\Lambda$-modules),
let $K^b(\proj\Lambda)$ (respectively, $K^b(\inj\Lambda)$) denote the homotopy category
of bounded complexes of $\proj\Lambda$ (respectively, $\inj\Lambda$) and let $\cd(\Mod\Lambda)$ (respectively,
$\cd^b(\mod\Lambda)$, $\cd^{-}(\mod\Lambda)$) denote the derived
category of $\Mod\Lambda$ (respectively, bounded derived category of
$\mod\Lambda$, bounded above derived category of $\mod\Lambda$). All these categories are triangulated with suspension functor the
shift functor. We view $\cd^{-}(\mod\Lambda)$ and
$\cd^b(\mod\Lambda)$ as triangulated subcategories of
$\cd(\Mod\Lambda)$.

The categories $\mod\Lambda$, $\cd^b(\mod \Lambda)$ and $K^b(\proj \Lambda)$ are Krull--Schmidt categories. An object $M$ of $\mod\Lambda$ (respectively, $\cd^b(\mod\Lambda)$, $K^b(\proj \Lambda)$) is said to be \emph{basic} if every indecomposable direct summand of $M$ has multiplicity $1$. The finite-dimensional algebra $\Lambda$ is said to be \emph{basic} if the free module of rank $1$ is basic in $\mod\Lambda$ (equivalently, in $\cd^b(\mod\Lambda)$ or $K^b(\proj\Lambda)$).

For a differential graded(=dg) algebra $A$, let $\cc(A)$ denote the
category of (right) dg modules over $A$ and $K(A)$ the homotopy category. Let $\cd(A)$ denote the
derived category of dg $A$-modules, \ie the triangle quotient of $K(A)$ by acyclic dg $A$-modules, \confer~\cite{Keller94,Keller06d}, and let
$\cd_{fd}(A)$ denote its full subcategory of dg $A$-modules whose
total cohomology is finite-dimensional. The category $\cc(A)$ is abelian and the other three categories are triangulated with suspension functor the shift functor of complexes. Let
$\per(A)=\thick(A_A)$, \ie the triangulated subcategory of $\cd(A)$
generated by the free dg $A$-module of rank $1$. 

 For two dg $A$-modules $M$ and $N$,  let $\cHom_A(M,N)$ denote the
complex whose degree $n$ component consists of those $A$-linear
maps from $M$ to $N$ which are homogeneous of degree $n$, and whose
differential  takes a homogeneous map $f$ of degree $n$ to \mbox{$d_N\circ
f-(-1)^n f\circ d_M$.} Then
\begin{eqnarray}
\Hom_{K(A)}(M,N)=H^0\cHom_A(M,N).\label{eq:morphism-homotopy-cat}
\end{eqnarray} 
A dg $A$-module $M$ is said to be \emph{$\ck$-projective} if  $\cHom_A(M,N)$ is acyclic when $N$ is an acyclic dg $A$-module. For example, $A_A$, the free dg $A$-module of rank $1$ is $\ck$-projective, because $\cHom_A(A,N)=N$. Dually, one defines \emph{$\ck$-injective} dg modules, and $D({}_AA)$ is $\ck$-injective. For two dg $A$-modules $M$ and $N$ such that $M$ is $\ck$-projective or $N$ is $\ck$-injective, we have 
\begin{eqnarray}
\Hom_{\cd(A)}(M,N)=\Hom_{K(A)}(M,N).
\label{eq:cofibrant}\end{eqnarray}
%Dually one defines \emph{fibrant} dg $A$-modules. See \cite[Section 3]{Keller94}. It follows from \cite[Theorem 3.1 and 3.2]{Keller94} that %the canonical quotient functor $K(A)\rightarrow \cd(A)$ admits a left adjoint functor $\mathbf{p}:\cd(A)\rightarrow K(A)$ (respectively, a right adjoint functor $\mathbf{i}:\cd(A)\rightarrow K(A)$) such that $\mathbf{p}M$ is cofibrant (respectively, $\mathbf{i}M$ is fibrant) for a dg $A$-module $M$. This $\mathbf{p}M$ is called a \emph{cofibrant} resolution 
%for any dg $A$-module $M$, there exists a cofibrant dg $A$-module $\mathbf{p}M$ (respectively, a fibrant dg $A$-module $\mathbf{i}M$) which is quasi-isomorphic to $M$, called the \emph{cofibrant resolution} (respectively, the \emph{fibrant resolution}) of $M$. Moreover, there are isomorphisms 
%\begin{eqnarray}
%\Hom_{K(A)}(M,\mathbf{i}N)\cong\Hom_{\cd(A)}(M,N)\cong\Hom_{K(A)}%(\mathbf{p}M,N),
%\label{eq:cofibrant}\end{eqnarray}
%which is natural in both $M$ and $N$.

Let $A$ and $B$ be
two dg algebras. Then a triangle equivalence between $\cd(A)$ and
$\cd(B)$ restricts to a triangle equivalence between $\per(A)$ and
$\per(B)$ and also to a triangle equivalence between $\cd_{fd}(A)$
and $\cd_{fd}(B)$. If $A$ is a finite-dimensional algebra viewed as
a dg algebra concentrated in degree 0, then $\cd(A)$ is exactly
$\cd(\Mod A)$, $\cd_{fd}(A)$ is $\cd^b(\mod A)$, $\per(A)$ is
triangle equivalent to $K^b(\proj A)$, and $\thick(\mathrm{D}(_A
A))$ is triangle equivalent to $K^b(\inj A)$.

\subsection{The Nakayama functor}\label{ss:nakayama-for-fd-alg}
Let $\Lambda$ be a finite-dimensional algebra. The \emph{Nakayama
functor} $\nu_{\mod\Lambda}$ is defined as
$\nu_{\mod\Lambda}=?\otimes_\Lambda \mathrm{D}(_\Lambda \Lambda)$,
and the \emph{inverse Nakayama functor} $\nu^{-1}_{\mod\Lambda}$ is
its right adjoint
$\nu^{-1}_{\mod\Lambda}=\Hom_{\Lambda}(\mathrm{D}(_\Lambda
\Lambda),?)$. They restrict to quasi-inverse equivalences between
$\proj\Lambda$ and $\inj\Lambda$.

The derived functors of $\nu_{\mod\Lambda}$ and
$\nu^{-1}_{\mod\Lambda}$, denoted by $\nu$ and $\nu^{-1}$, restrict
to quasi-inverse triangle equivalences  between
$K^b(\proj\Lambda)$ and $K^b(\inj\Lambda)$. When $\Lambda$ is
self-injective, they restrict to quasi-inverse triangle
auto-equivalences of $\cd^b(\mod\Lambda)$.

The Auslander--Reiten formula for $M$ in $K^b(\proj\Lambda)$
and $N$ in $\cd(\Mod\Lambda)$ (\confer~\cite[Chapter 1, Section
4.6]{Happel88}) provides an isomorphism
\[\mathrm{D}\Hom(M,N)\cong\Hom(N,\nu M),\]
which is natural in $M$ and $N$. When $K^b(\proj\Lambda)$ coincides
with $K^b(\inj\Lambda)$ (that is, when $\Lambda$ is Gorenstein), it
has Auslander--Reiten triangles and the Auslander--Reiten
translation is $\tau=\nu\circ\Sigma^{-1}$.

\section{The four concepts}\label{s:the-four-str}
\label{section-fourconcepts}

In this section we introduce silting objects, simple-minded collections, $t$-structures and co-$t$-structure. Let $\cc$ be a triangulated category with suspension functor
$\Sigma$.

\subsection{Silting objects}\label{ss:silting-obj}

A subcategory $\cm$ of $\cc$ is called a \emph{silting
subcategory}~\cite{KellerVossieck88,AiharaIyama12} if it is stable
for taking direct summands and generates $\cc$ (\ie
$\cc=\thick(\cm)$) and if $\Hom(M,\Sigma^m N)=0$ for $m>0$ and
$M,N\in\cm$.

\begin{theorem}\label{t:silting-have-same-number-of-summands}\emph{(\cite[Theorem {2.27}]{AiharaIyama12})}
Assume that $\cc$ is Krull--Schmidt and has a silting subcategory
$\cm$. Then the Grothendieck group of $\cc$ is free and its rank is
equal to the cardinality of the set of isomorphism classes of
indecomposable objects of $\cm$.
\end{theorem}

An object $M$ of $\cc$ is called a \emph{silting object} if $\add M$
is a silting subcategory of $\cc$. This notion was introduced by Keller and Vossieck in~\cite{KellerVossieck88} to study $t$-structures on the bounded derived category of representations over a Dynkin quiver. Recently it has also been studied by Wei~\cite{Wei11} (who uses the terminology \emph{semi-tilting complexes}) from the perspective of classical tilting theory. A \emph{tilting object} is a silting object $M$ such that $\Hom(M,\Sigma^m M)=0$ for $m<0$. For an algebra
$\Lambda$, a tilting object in $K^b(\proj \Lambda)$ is called a
\emph{tilting complex} in the literature. For example, the free
module of rank $1$ is a tilting object in $K^b(\proj \Lambda)$.
Assume that $\Lambda$ is finite-dimensional.
Theorem~\ref{t:silting-have-same-number-of-summands} implies that (a)
any silting subcategory of $K^b(\proj\Lambda)$ is the additive
closure of a silting object, and (b) any two basic silting objects
have the same number of indecomposable direct summands. We will
rederive (b) as a corollary of the existence of a certain
derived equivalence
(Corollary~\ref{c:silting-have-fixed-number-of-summands}).

\subsection{Simple-minded collections}\label{ss:simple-minded-collection}

\begin{definition} A collection $X_1,\ldots,X_r$ of objects of $\cc$ is said to be
\emph{simple-minded} (cohomologically Schurian
in~\cite{Al-Nofayee09}) if the following conditions hold for
$i,j=1,\ldots,r$
\begin{itemize}
\item[$\cdot$] $\Hom(X_i,\Sigma^m X_j)=0,~~\forall~m<0$,
\item[$\cdot$] {$\End(X_i)$ is a division algebra and $\Hom(X_i,X_j)$ vanishes for $i\neq j$,}
%$\Hom(X_i,X_j)=\begin{cases} K & \text{if\ }i=j,\\
                          %                 0 & \text{otherwise},
                          %                 \end{cases}$
\item[$\cdot$] $X_1,\ldots,X_r$ generate
$\cc$ (\ie $\cc=\thick(X_1,\ldots,X_r)$).
\end{itemize}
\end{definition}
Simple-minded collections {are variants of simple-minded systems in~\cite{KoenigLiu10}} and were first studied by
Rickard~\cite{Rickard02} in the context of derived equivalences of
symmetric algebras. For a finite-dimensional algebra $\Lambda$, a
complete collection of {pairwise} non-isomorphic simple modules is a
simple-minded collection in $\cd^b(\mod \Lambda)$. {A natural question is: do any two simple-minded collections have the same collection of endomorphism algebras?}

\subsection{t-structures}\label{ss:t-str}

A \emph{$t$-structure} on
$\cc$ (\cite{BeilinsonBernsteinDeligne82}) is a pair $(\cc^{\leq
0},\cc^{\geq 0})$ of strict (that is, closed under
isomorphisms) and full  subcategories of $\cc$ such that
\begin{itemize}
\item[$\cdot$] $\Sigma\cc^{\leq 0}\subseteq\cc^{\leq 0}$ and
$\Sigma^{-1}\cc^{\geq 0}\subseteq\cc^{\geq 0}$;
\item[$\cdot$] $\Hom(M,\Sigma^{-1}N)=0$ for $M\in\cc^{\leq 0}$
and $N\in\cc^{\geq 0}$,
\item[$\cdot$] for each $M\in\cc$ there is a triangle $M'\rightarrow
M\rightarrow M''\rightarrow\Sigma M'$ in $\cc$ with $M'\in\cc^{\leq
0}$ and $M''\in\Sigma^{-1}\cc^{\geq 0}$.
\end{itemize}
The two subcategories $\cc^{\leq 0}$ and $\cc^{\geq 0}$ are often
called the \emph{aisle} and the \emph{co-aisle} of the $t$-structure
respectively. The \emph{heart} $\cc^{\leq 0}~\cap~\cc^{\geq 0}$ is
always abelian. Moreover, $\Hom(M,\Sigma^m N)$ vanishes for any two
objects $M$ and $N$ in the heart and for any $m<0$. The
$t$-structure $(\cc^{\leq 0},\cc^{\geq 0})$ is said to be
\emph{bounded} if
$$\bigcup_{n\in\mathbb{Z}}\Sigma^n \cc^{\leq
0}=\cc=\bigcup_{n\in\mathbb{Z}}\Sigma^n\cc^{\geq 0}.$$ 
A bounded $t$-structure is one of the two ingredients of a Bridgeland stability condition~\cite{Bridgeland07}. A typical
example of a $t$-structure is the pair $(\cd^{\leq 0},\cd^{\geq 0})$
for the derived category $\cd(\Mod \Lambda)$ of an (ordinary)
algebra $\Lambda$, where $\cd^{\leq 0}$ consists of complexes with
vanishing cohomologies in positive degrees, and $\cd^{\geq 0}$
consists of complexes with vanishing cohomologies in negative
degrees. This $t$-structure restricts to a bounded $t$-structure of
$\cd^b(\mod \Lambda)$ whose heart is $\mod\Lambda$, which is a
\emph{length category}, \ie every object in it has finite length.
The following lemma is well-known.

\begin{lemma}\label{l:heart} Let $(\cc^{\leq 0},\cc^{\geq 0})$ be a bounded $t$-structure on $\cc$ with heart $\ca$.
\begin{itemize}
 \item[(a)] The embedding $\ca\rightarrow\cc$ induces an isomorphism $K_0(\ca)\rightarrow K_0(\cc)$ of Grothendieck groups.
% \item[(b)] The embedding $\ca\rightarrow\cc$ extends to a triangle functor $R:\cd^b(\ca)\rightarrow \cc$ such that
%$R(M,N):\Hom_{\cd^b(\ca)}(M,\Sigma N)\rightarrow\Hom_{\cc}(R(M),\Sigma R(N))$ is bijective.
 \item[(b)] $\cc^{\leq 0}$ respectively $\cc^{\geq 0}$ is the extension
closure of $\Sigma^m\ca$ for $m\geq 0$ respectively for $m\leq 0$.
 \item[(c)] $\cc=\thick(\ca)$.
\end{itemize}
Assume further $\ca$ is a length category
with simple objects $\{S_i\mid i\in I\}$.
\begin{itemize}
\item[(d)] $\cc^{\leq 0}$
respectively $\cc^{\geq 0}$ is the extension closure of
$\Sigma^m\{S_i\mid i\in I\}$ for $m\geq 0$ respectively for $m\leq 0$.
\item[(e)] $\cc=\thick(S_i,i\in I)$.
\item[(f)] If $I$ is finite, then $\{S_i\mid i\in I\}$ is a simple-minded collection.
\end{itemize}
\end{lemma}

\subsection{Co-t-structures}\label{ss:co-t-str}

According to \cite{Pauksztello08}, a \emph{co-$t$-structure} on
$\cc$  (or \emph{weight structure} in~\cite{Bondarko10}) is a pair $(\cc_{\geq 0},\cc_{\leq
0})$ of strict and full subcategories of $\cc$ such that
\begin{itemize}
\item[$\cdot$] both $\cc_{\geq 0}$ and $\cc_{\leq
0}$ are additive and closed under taking direct summands,
\item[$\cdot$] $\Sigma^{-1}\cc_{\geq 0}\subseteq\cc_{\geq 0}$ and
$\Sigma\cc_{\leq 0}\subseteq\cc_{\leq 0}$;
\item[$\cdot$] $\Hom(M,\Sigma N)=0$ for $M\in\cc_{\geq 0}$
and $N\in\cc_{\leq 0}$,
\item[$\cdot$] for each $M\in\cc$ there is a triangle $M'\rightarrow
M\rightarrow M''\rightarrow\Sigma M'$ in $\cc$ with $M'\in\cc_{\geq
0}$ and $M''\in\Sigma\cc_{\leq 0}$.
\end{itemize}
The \emph{co-heart} is defined as the intersection $\cc_{\geq
0}~\cap~\cc_{\leq 0}$. This is usually not an abelian category. For
any two objects $M$ and $N$ in the co-heart, the morphism space $\Hom(M,\Sigma^m
N)$ vanishes for any $m>0$. The co-$t$-structure $(\cc^{\leq 0},\cc^{\geq
0})$ is said to be \emph{bounded}~\cite{Bondarko10} if
$$\bigcup_{n\in\mathbb{Z}}\Sigma^n \cc_{\leq
0}=\cc=\bigcup_{n\in\mathbb{Z}}\Sigma^n\cc_{\geq 0}.$$ 
A bounded co-$t$-structure is one of the two ingredients of a J{\o}rgensen--Pauksztello costability condition~\cite{JorgensenPauksztello11}.
A typical
example of a co-$t$-structure is the pair $(K_{\geq 0},K_{\leq 0})$
for the homotopy category $K^b(\proj \Lambda)$ of a
finite-dimensional algebra $\Lambda$, where $K_{\geq 0}$ consists of
complexes which are homotopy equivalent to a complex bounded below at
$0$, and $K_{\leq 0}$ consists of complexes which are homotopy
equivalent to a complex bounded above at $0$. The co-heart of this
co-$t$-structure is $\proj\Lambda$.
%The following lemma is easy.

\begin{lemma}\label{l:co-heart}\emph{{(\cite[Theorem 4.10 (a)]{MendozaSaenzSantiagoSouto10})}}
Let $(\cc_{\geq 0},\cc_{\leq 0})$ be a bounded co-$t$-structure on
$\cc$ with co-heart $\ca$. Then $\ca$ is a silting subcategory of
$\cc$.
\end{lemma}
\begin{proof} {For the convenience of the reader we give a proof.} It suffices to show that $\cc=\thick(\ca)$.
Let $M$ be an object of $\cc$. Since the co-$t$-structure is
bounded, there are integers $m\geq n$ such that
$M\in\Sigma^m\cc_{\geq 0}\cap\Sigma^n\cc_{\leq 0}$. Up to suspension
and cosuspension we may assume that $m=0$. If $n=0$, then $M\in\ca$.
Suppose $n<0$. There exists a triangle
\[\xymatrix{M'\ar[r]&M\ar[r]&M''\ar[r]&\Sigma M'}\]
with $M'\in\Sigma^{-1}\cc_{\geq 0}$ and $M''\in\cc_{\leq 0}$. In
fact, $M''\in\ca$, see~\cite[Proposition 1.3.3.6]{Bondarko10}.
Moreover, $\Sigma M'\in\Sigma^{n+1}\cc_{\leq 0}$ due to the triangle
\[\xymatrix{M''\ar[r]&\Sigma M'\ar[r]&\Sigma M\ar[r]&\Sigma M''}\]
since both $M''$ and $\Sigma M$ belong to $\Sigma^{n+1}\cc_{\leq 0}$
and $\cc_{\leq 0}$ is extension closed (see~\cite[Proposition
1.3.3.3]{Bondarko10}). So $\Sigma M'\in\cc_{\geq
0}\cup\Sigma^{n+1}\cc_{\leq 0}$. We finish the proof by induction on
$n$.
\end{proof}

\begin{proposition}\label{p:AI's-from-silting-to-co-t-str}
\emph{(\cite[Proposition 2.22]{AiharaIyama12}, \cite[(proof
of) Theorem 4.3.2]{Bondarko10},~\cite[Theorem 5.5]{MendozaSaenzSantiagoSouto10} and~\cite{KellerNicolas11})}
Let $\ca$ be a silting subcategory of $\cc$. Let $\cc_{\leq 0}$
respectively $\cc_{\geq 0}$ be the extension closure of
$\Sigma^m\ca$ for $m\geq 0$ respectively for $m\leq 0$. Then
$(\cc_{\geq 0},\cc_{\leq 0})$ is a bounded co-$t$-structure on $\cc$
with co-heart $\ca$.
\end{proposition}

\section{Finite-dimensional  non-positive dg algebras}\label{a:fin-dim-negative-dg} \label{section-dgalgebras}
In this section we study derived categories of \emph{non-positive dg algebras}, \ie dg algebras $A=\bigoplus_{i\in\mathbb{Z}}A^i$ with $A^i=0$ for $i>0$, especially finite-dimensional non-positive dg algebras, \ie , {non-positive dg algebras which, as vector spaces, are finite-dimensional. These results will be used in Sections~\ref{ss:silting-der-equiv} and \ref{s:from-silting}.}

\medskip
Non-positive dg
algebras are closely related to silting objects.
A triangulated category is said to be \emph{algebraic} if it is triangle equivalent to the stable category of a Frobenius category.

\begin{lemma}\label{l:negative-dg-and-silting}
\begin{itemize}
\item[(a)]
 Let $A$ be a non-positive dg algebra. The free dg $A$-module of rank $1$ is a silting object of $\per(A)$.
\item[(b{)}]
 Let $\cc$ be an algebraic triangulated category with split idempotents
and let $M\in\cc$ be a silting object. Then there is a non-positive dg
algebra $A$ together with a triangle equivalence
$\per(A)\stackrel{\sim}{\rightarrow} \cc$ which takes $A$ to $M$.
\end{itemize}
\end{lemma}
\begin{proof} (a)
 This is because $\Hom_{\per(A)}(A,\Sigma^i A)=H^i(A)$ vanishes for
 $i>0$.

 (b) By~\cite[Theorem 3.8 b)]{Keller06d} (which is a `classically generated' version of~\cite[Theorem 4.3]{Keller94}), there is a dg algebra $A'$ together with a triangle equivalence
$\per(A')\stackrel{\sim}{\rightarrow}\cc$. In particular, there are isomorphisms $\Hom_{\per(A')}(A',\Sigma^i A')\cong\Hom_{\cc}(M,\Sigma^i M)$
for all $i\in\mathbb{Z}$. Since $M$ is a silting object, $A'$ has vanishing cohomologies in positive degrees. Therefore, if $A=\tau_{\leq 0}A'$
is the standard truncation at position $0$, then the embedding $A\hookrightarrow A'$ is a quasi-isomorphism. {It follows that} there is a composite
triangle equivalence
\[\xymatrix{\per(A)\ar[r]^{\sim}&\per(A')\ar[r]^(0.6){\sim}&\cc}\]
which takes $A$ to $M$.
\end{proof}

In the sequel of this section we assume that $A$ is a
finite-dimensional non-positive dg algebra.
 The 0-th cohomology
$\bar{A}=H^{0}(A)$ of $A$ is a finite-dimensional $K$-algebra. Let $\Mod
\bar{A}$ and $\mod\bar{A}$ denote the category of (right) modules
over $\bar{A}$ and its subcategory consisting of those
finite-dimensional modules. Let $\pi:A\rightarrow \bar{A}$ be the
canonical projection. We view $\Mod \bar{A}$ as a subcategory of
$\cc(A)$ via $\pi$.

The total cohomology $H^{*}(A)$ of $A$ is a finite-dimensional
graded algebra with multiplication induced from the multiplication
of $A$. Let $M$ be a dg $A$-module. Then the total cohomology
$H^{*}(M)$ carries a graded $H^{*}(A)$-module structure, and hence a
graded $\bar{A}=H^{0}(A)$-module structure. In particular, a stalk
dg $A$-module concentrated in degree 0 is an $\bar{A}$-module.

\subsection{The standard $t$-structure}\label{s:standard-t-structure} We follow~\cite{FuMasterThesis,Amiot09,KellerYang11}, where the dg algebra is not necessarily
finite-dimensional.

Let $M=\ldots\rightarrow M^{i-1}\stackrel{d^{i-1}}{\rightarrow}
M^i\stackrel{d^i}{\rightarrow}M^{i+1}\rightarrow\ldots$ be a dg
$A$-module. Consider the standard truncation functors $\tau_{\leq
0}$ and $\tau_{> 0}$:
\[\xymatrix@R=1pc{\tau_{\leq 0}M =\\ \tau_{> 0}M=}
\xymatrix@C=1pc@R=1pc{  \ldots \ar[r] &
M^{-2 }\ar[r]^{d^{-2}}& M^{-1}\ar[r]^{d^{-1}}&\mathrm{ker}d^{0}\ar[r]&
0\ar[r]&0\ar[r]&0\ar[r]&\ldots\\
\ldots\ar[r]& 0\ar[r]& 0\ar[r]&
M^0/\mathrm{ker} d^0\ar[r]^(0.65){d^0}&M^{1}\ar[r]^{d^{1}}&M^2\ar[r]^{d^2}&M^3\ar[r]&\ldots}
\]
Since $A$ is non-positive, $\tau_{\leq 0}M$ is a dg $A$-submodule of
$M$ and $\tau_{>0}M$ is the corresponding quotient dg $A$-module.
Hence there is a distinguished triangle in $\mathcal{D}(A)$
\[\tau_{\leq 0}M\rightarrow M\rightarrow \tau_{>0}M\rightarrow
\Sigma\tau_{\leq 0}M.\] These two functors define a $t$-structure
$(\mathcal{D}^{\leq 0},\mathcal{D}^{\geq 0})$ on $\mathcal{D}(A)$,
where $\mathcal{D}^{\leq 0}$ is the subcategory of $\mathcal{D}(A)$
consisting of dg $A$-modules with vanishing cohomology in positive
degrees, and $\mathcal{D}^{\geq 0}$ is the subcategory of
$\mathcal{D}(A)$ consisting of dg $A$-modules with vanishing
cohomology in negative degrees.

By the definition of the $t$-structure
$(\mathcal{D}^{\leq 0},\mathcal{D}^{\geq 0})$, the heart
$\mathcal{H}=\mathcal{D}^{\leq 0}\cap\mathcal{D}^{\geq 0}$ consists
of those dg $A$-modules whose cohomology is concentrated in degree
0. Thus the functor $H^0$ induces an equivalence
\begin{eqnarray*}H^0:\mathcal{H}&\longrightarrow&\Mod\bar{A}.\\
M&\mapsto& H^{0}(M)
\end{eqnarray*}
See also~\cite[Theorem 1.3]{HoshinoKatoMiyachi02}. The $t$-structure
$(\mathcal{D}^{\leq 0},\mathcal{D}^{\geq 0})$ on $\mathcal{D}(A)$
restricts to a bounded $t$-structure on $\mathcal{D}_{fd}(A)$ with
heart equivalent to $\mod\bar{A}$.

%{\color{blue}
%\begin{lemma}
%Let $M,N$ be dg $A$-modules such that $H^i(N)=0$ for $i<1$. Then the morphism $M\rightarrow\tau_{>0}M$ induces an isomorphism $\Hom(\tau_{>0}M,N)\stackrel{\simeq}{\longrightarrow} \Hom(M,N)$.
%\end{lemma}
%}

\subsection{Morita reduction}
 Let $d$ be the differential of $A$. Then $d(A^0)=0$.

 Let $e$ be an idempotent of $A$. For degree reasons, $e$ must belong to
 $A^0$, and the graded
 subspace $eA$ of $A$ is a dg submodule: $d(ea)=d(e)a+ed(a)=ed(a)$.
 Therefore for each
 decomposition $1=e_1+\ldots+e_n$ of the unity into a sum of primitive
 orthogonal idempotents, there is a direct sum decomposition
 $A=e_1A\oplus\ldots\oplus e_nA$ of $A$ into indecomposable dg
 $A$-modules.
 Moreover, if $e$ and $e'$ are two idempotents of $A$ such that
 $eA\cong e'A$ as ordinary modules over the ordinary algebra $A$,
 then this isomorphism is also an isomorphism of dg modules. Indeed,
 there are two elements of $A$ such that $fg=e$ and $gf=e'$. Again
 for degree reasons, $f$ and $g$ belong to $A^0$. So they induce
 isomorphisms of dg $A$-modules: $eA\rightarrow e'A$, $a\mapsto ga$
 and $e'A\rightarrow eA$, $a\mapsto fa$.  It follows that the above
 decomposition of $A$ into a direct sum of
indecomposable dg modules is essentially unique. Namely, if
 $1=e'_1+\ldots+e'_n$ is another decomposition of the unity into a
 sum of primitive orthogonal idempotents, then $m=n$ and up to
 reordering, $e_1A\cong e'_1A$, $\ldots$, $e_nA\cong e'_nA$.

\subsection{The perfect derived category}\label{ss:perfect-derived-cat}

Since $A$ is finite-dimensional (and thus has finite-dimensional total
cohomology), $\per(A)$ is a triangulated
subcategory of $\mathcal{D}_{fd}(A)$.

We assume, as we may, that $A$ is basic. Let $1=e_1+\ldots+e_n$ be a
decomposition of $1$ in $A$ into a sum of primitive orthogonal
idempotents. Since $d(x)=\lambda_1 e_{i_1}+\ldots+\lambda_s e_{i_s}$
implies that $d(e_{i_j}x)=\lambda_j e_{i_j}$, the
intersection of the space spanned by $e_1,\ldots,e_n$ with the image
of the differential $d$ has a basis consisting of some $e_i$'s, say
$e_{r+1},\ldots,e_n$. So, $e_{r+1} A,\ldots,e_n
A$ are homotopic to zero.

We say that a dg $A$-module $M$ is \emph{strictly perfect} if its underlying graded module is of the form $\bigoplus_{j=1}^{N}R_j$, where $R_j$ belongs to $\add(\Sigma^{t_j}A)$ for some $t_j$ with $t_1<t_2<\ldots<t_N$, and if its differential is of the form $d_{int}+\delta$, where $d_{int}$ is the direct sum of the differential of the $R_j$'s, and $\delta$, as a degree $1$ map from $\bigoplus_{j=1}^{N}R_j$ to itself, is a strictly upper triangular matrix whose entries are in $A$. It is \emph{minimal} if in addition no shifted copy of $e_{r+1}A,\ldots, e_n A$ belongs to $\add(R_1,\ldots,R_j)$, and the entries of $\delta$ are in the radical of $A$, \confer~\cite[Section 2.8]{Plamondon11}. Strictly perfect dg modules are $\ck$-projective. If $A$ is an ordinary algebra, then strictly perfect dg modules are precisely bounded complexes of finitely generated projective modules.
%{\color{blue} A dg $A$-module $M$ is said to \emph{have property (FP)} if its underlying graded module is of the form $\bigoplus_{j\in\mathbb{Z}}R_j$, where $R_j$ belongs to $\add(\Sigma^{j}A)$ and $R_j$ vanishes for sufficiently small $j$, and
%its differential is of the form $d_{int}+\delta$, where $d_{int}$ is the direct sum of the differential of the $R_j$'s, and $\delta$ is a degree $1$ map from $\bigoplus_{j\in\mathbb{Z}}R_j$ to itself. Due to the assumption that $A$ is finite-dimensional, each component of $M$ is finite-dimensional. Due to the assumption that $A$ is non-positive, the map $\delta$ is a strictly lower triangular matrix whose entries are in $A$. If follows that $M_j:=\bigoplus_{j'\leq j}R_{j'}$ is a dg submodule of $M$. It is easily seen that (up to a suitable shift in the index) the filtration
%\[\ldots\subseteq M_{-1}\subseteq M_0\subseteq M_1\subseteq\ldots\subseteq M\]
%satisfies the conditions (1), (2) and (3) in the definition of a cofibrant dg module. Thus $M$ is cofibrant.

%A dg $A$-module having property (FP) is \emph{minimal} if  no shifted copy of $e_{r+1}A,\ldots, e_n A$ belongs to $\add(R_j,~j\in\mathbb{Z})$, and the entries of $\delta$ are in the radical of $A$. It is \emph{strictly perfect} if all but finitely many $R_j$ vanishes. If $A$ is an ordinary algebra, then dg $A$-modules having property (FP) are precisely bounded above conplexes of finitely generated projective modules and strictly perfect dg modules are precisely bounded complexes of finitely generated projective modules.
%}

\begin{lemma}\label{l:minimal-perfect-resolution} Let $M$ be a
dg $A$-module belonging to $\per(A)$. Then $M$ is quasi-isomorphic
to a minimal strictly perfect dg $A$-module.
\end{lemma}
\begin{proof} Bearing in mind that $e_1 A,\ldots, e_r A$ have local
 endomorphism algebras and $e_{r+1}A,\ldots, e_n A$ are homotopic to zero,
 we prove the assertion as in~\cite[Lemma 2.14]{Plamondon11}. \end{proof}

\subsection{Simple modules}\label{ss:simpledg} Assume
that $A$ is basic. 
%and that the base field $K$ is algebraically closed.
According to the preceding subsection, we may assume that there is a
decomposition $1=e_1+\ldots+e_r+e_{r+1}+\ldots+e_n$ of the unity of
$A$ into a sum of primitive orthogonal idempotents such that
$1=\bar{e}_1+\ldots+\bar{e}_r$ is a decomposition of $1$ in
$\bar{A}$ into a sum of primitive orthogonal idempotents.

Let $S_1,\ldots,S_r$ be a complete set of pairwise non-isomorphic simple
$\bar{A}$-modules {and let $R_1,\ldots,R_r$ be their endomorphism algebras}. Then
\[\mathcal{H}om_{A}(e_i A,S_{j})=\begin{cases}  {}_{R_j}R_j & \text{if }i=j,\\
                                           0 & \text{otherwise}.
                                           \end{cases}\]
Therefore, by (\ref{eq:morphism-homotopy-cat}) and (\ref{eq:cofibrant}),
\[\Hom_{\cd(A)}(e_i A,\Sigma^m S_j)=\begin{cases} {{}_{R_j}R_j} & \text{ if } i=j \text{ and } m=0,\\ 0 & \text{ otherwise}.\end{cases}\]
Moreover, $\{e_1 A,\ldots,e_rA\}$ and $\{S_1,\ldots,S_r\}$
characterise each other by this property. On the one hand, if $M$ is
a dg $A$-module such that for some integer $1\leq j\leq r$
\[\Hom_{\cd(A)}(e_i A,\Sigma^m M)=\begin{cases} {{}_{R_j}R_j} & \text{ if } i=j \text{ and } m=0,\\ 0 & \text{ otherwise},\end{cases}\]
 then $M$ is isomorphic in $\cd(A)$ to $S_j$.
On the other hand, let $M$ be an object of $\per(A)$ such that for
some integer $1\leq i\leq r$
\[\Hom_{\cd(A)}(M,\Sigma^m S_j)=\begin{cases} {{}_{R_j}R_j} & \text{ if } i=j \text{ and } m=0,\\ 0 & \text{ otherwise}.\end{cases}\]
Then by replacing $M$ by its minimal perfect resolution
(Lemma~\ref{l:minimal-perfect-resolution}), we see that $M$ is
isomorphic in $\cd(A)$ to $e_i A$.

Further, recall from Section~\ref{s:standard-t-structure} that
$\cd_{fd}(A)$ admits a standard $t$-structure whose heart is
equivalent to $\mod\bar{A}$. This implies that the simple modules
$S_1,\ldots,S_r$ form a simple-minded collection in $\cd_{fd}(A)$.

\subsection{The Nakayama functor}

For a complex $M$ of $K$-vector spaces, we define its dual as
$\mathrm{D}(M)=\mathcal{H}om_{K}(M,K)$, where $K$ in the second
argument is considered as a complex concentrated in degree 0. One
checks that $\mathrm{D}$ defines a duality between
finite-dimensional dg $A$-modules and finite-dimensional dg
$A^{op}$-modules.

Let $e$ be an idempotent of $A$ and $M$ a dg $A$-module. Then
there is a canonical isomorphism
\[\mathcal{H}om_{A}(eA,M)\cong Me.\]
If in addition each component of $M$ is finite-dimensional , there are
canonical isomorphisms
\[\mathcal{H}om_{A}(eA,M)\cong Me\cong \mathrm{D}\mathcal{H}om_{A}(M,\mathrm{D}(Ae)).\]

Let $\cc(A)$ denote the category of dg $A$-modules. The
Nakayama functor $\nu:\cc(A)\rightarrow\cc(A)$ is defined by
$\nu(M)=\mathrm{D}\mathcal{H}om_{A}(M,A)$ \cite[Section
10]{Keller94}. There are canonical isomorphisms
\[\mathrm{D}\cHom_A(M,N)\cong \cHom_A(N,\nu M)\] for any strictly perfect dg
$A$-module $M$ and any dg $A$-module $N$. Then
$\nu(eA)=\mathrm{D}(Ae)$ for an idempotent $e$ of $A$, and the
functor $\nu$ induces a triangle equivalences between the
subcategories $\per(A)$ and $\thick(\mathrm{D}(A))$ of {$\cd(A)$} with
quasi-inverse given by
$\nu^{-1}(M)=\mathcal{H}om_{A}(\mathrm{D}(A),M)$. Moreover, we have the Auslander--Reiten formula
\[\mathrm{D}\Hom(M,N)\cong\Hom(N,\nu M),\]
which is natural in $M\in\per(A)$ and $N\in\cd(A)$.

 Let
$e_1,\ldots,e_r$, $S_1,\ldots,S_r$ {and $R_1,\ldots,R_r$} be as in the preceding
subsection. Then
\[\mathcal{H}om_{A}(S_j, \mathrm{D}(Ae_i))\cong \mathrm{D}\cHom_A(e_i A, S_j)=\begin{cases} {(R_j)_{R_j}} & \text{if }i=j,\\
                                           0 & \text{otherwise}.
                                           \end{cases}\]
Therefore, by (\ref{eq:morphism-homotopy-cat}) and (\ref{eq:cofibrant}),
\[
\Hom_{\cd(A)}(S_j,\Sigma^m\mathrm{D}(Ae_i))
=\begin{cases} {(R_j)_{R_j}} & \text{ if } i=j \text{ and } m=0,\\
0 & \text{ otherwise}.\end{cases}
\] 
Moreover,
$\{\mathrm{D}(Ae_1),\ldots,\mathrm{D}(Ae_r)\}$ and
$\{S_1,\ldots,S_r\}$ characterise each other in $\mathcal{D}(A)$ by
this property. This follows from the arguments in the preceding
subsection by applying the functors $\nu$ and $\nu^{-1}$.

\subsection{The standard co-$t$-structure}

Let $\cp_{\leq 0}$ (respectively, $\cp_{\geq 0}$) be the smallest
full subcategory of $\per(A)$ containing $\{\Sigma^m A\mid m\geq
0\}$ (respectively, $\{\Sigma^m A\mid m\leq 0\}$) and closed
under taking extensions and direct summands. The following lemma
is a special case of Proposition~\ref{p:AI's-from-silting-to-co-t-str}.
%~\cite[Proposition6.2.1]{Bondarko10}.
For the convenience of the reader we include a proof.

\begin{lemma}\label{l:standard-co-t-str}
 The pair $(\cp_{\geq 0},\cp_{\leq 0})$ is a co-$t$-structure on $\per(A)$.
Moreover, its co-heart is $\add(A_A)$.
\end{lemma}

\begin{proof} Since $\Hom(A,\Sigma^m A)=0$ for $m\geq 0$, it follows
that $\Hom(X,\Sigma Y)=0$ for $M\in\cp_{\geq 0}$ and $N\in\cp_{\leq 0}$.
It remains to show that any object $M$ in $\per(A)$ fits into a triangle
whose outer terms belong to $\cp_{\geq 0}$ and $\cp_{\leq 0}$, respectively.
By Lemma~\ref{l:minimal-perfect-resolution}, we may assume that $M$ is minimal
perfect. Write $M=(\bigoplus_{j=1}^{N}R_j,d_{int}+\delta)$ as in
Section~\ref{ss:perfect-derived-cat}. Let $N'\in\{1,\ldots,N\}$ be the unique
integer such that $t_{N'}\geq 0$ but $t_{N'+1}<0$. Let $M'$ be the graded module
$\bigoplus_{j=1}^{N'}R_j$ endowed with the differential restricted
from $d_{int}+\delta$. Because $d_{int}+\delta$ is upper triangular,
$M'$ is a dg submodule of $M$. Clearly $M'$ belongs to $\cp_{\geq 0}$ and the
quotient $M''=M/M'$ belongs to $\Sigma\cp_{\leq 0}$. Thus we obtain
the desired triangle
\[\xymatrix{M'\ar[r]&M\ar[r]&M''\ar[r]&\Sigma M'}\]
with $M'$ in $\cp_{\geq 0}$ and $M''$ in $\Sigma\cp_{\leq 0}$.
\end{proof}

\section{The maps}\label{s:the-maps}
\label{section-maps}

Let $\Lambda$ be a finite-dimensional basic $K$-algebra.
This section is devoted to
defining the maps in the following diagram.

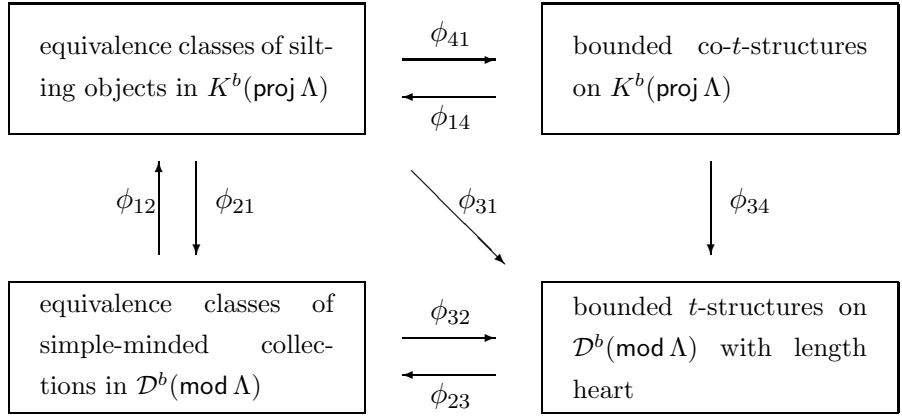
\[{\setlength{\unitlength}{0.7pt}
\begin{picture}(500,250)
\put(285,180){\framebox(190,70){\parbox{110pt}{\small bounded
co-$t$-structures on $K^b(\proj\Lambda)$}}}

\put(0,180){\framebox(190,70){\parbox{110pt}{\small equivalence
classes of silting objects in $K^b(\proj\Lambda)$}}}

\put(0,30){\framebox(190,70){\parbox{110pt}{\small equivalence
classes of simple-minded collections in $\cd^b(\mod\Lambda)$}}}

\put(285,30){\framebox(190,70){\parbox{110pt}{\small bounded
$t$-structures on $\cd^b(\mod\Lambda)$ with length heart}}}

\put(210,220){$\vector(1,0){50}$} \put(260,200){$\vector(-1,0){50}$}
\put(225,230){$\phi_{41}$} \put(225,185){$\phi_{14}$}
\put(210,70){$\vector(1,0){50}$} \put(260,50){$\vector(-1,0){50}$}
\put(225,80){$\phi_{32}$} \put(225,35){$\phi_{23}$}
\put(375,165){$\vector(0,-1){50}$} \put(385,140){$\phi_{34}$}
\put(100,165){$\vector(0,-1){50}$} \put(80,115){$\vector(0,1){50}$}
\put(57,140){$\phi_{12}$} \put(110,140){$\phi_{21}$}
\put(215,160){$\vector(1,-1){50}$} \put(240,140){$\phi_{31}$}
\end{picture}}\]

\subsection{Silting objects induce derived
equivalences}\label{ss:silting-der-equiv}

Let $M$ be a basic silting object of the category $K^b(\proj \Lambda)$. By definition, $M$ is a bounded complex of finitely generated projective $\Lambda$-modules such that $\Hom_{K^b(\proj\Lambda)}(M,\Sigma^m M)$ vanishes for all $m>0$. By
Lemma~\ref{l:negative-dg-and-silting}, there is a non-positive dg
algebra whose perfect derived category is triangle equivalent to
$K^b(\proj\Lambda)$. This equivalence sends the free dg module of
rank $1$ to $M$. Below we explicitly construct such a dg
algebra.

Consider $\cHom_\Lambda(M,M)$.
Recall that the degree $n$ component of $\cHom_\Lambda(M,M)$ consists
of those $\Lambda$-linear maps from $M$ to itself which are
homogeneous of degree $n$. The differential of
$\cHom_\Lambda(M,M)$ takes a homogeneous map $f$ of degree $n$ to
$d\circ f-(-1)^n f\circ d$, where $d$ is the differential of $M$. This differential and the composition of maps makes $\cHom_\Lambda(M,M)$ into a dg algebra.
Therefore $\cHom_\Lambda(M,M)$ is a finite-dimensional dg algebra. Moreover,
$H^m(\cHom_\Lambda(M,M))=\Hom_{\cd(\Lambda)}(M,\Sigma^m M)$ for any integer
$m$, by (\ref{eq:morphism-homotopy-cat}) and (\ref{eq:cofibrant}). Because $M$ is a silting object, $\cHom_\Lambda(M,M)$ has cohomology
concentrated in non-positive degrees. Take the
truncated dg algebra $\tilde\Gamma=\tau_{\leq 0}\cHom_\Lambda(M,M)$,
where $\tau_{\leq 0}$ is the standard truncation at position $0$.
Then the embedding $\tilde\Gamma\rightarrow\cHom_\Lambda(M,M)$ is a
quasi-isomorphism {of dg algebras}, and hence $\tilde\Gamma$ is a finite-dimensional
non-positive dg algebra. Therefore, the derived category
$\cd(\tilde\Gamma)$ carries a natural $t$-structure $(\cd^{\leq 0},
\cd^{\geq 0})$ with heart $\cd^{\leq 0}\cap\cd^{\geq 0}$ equivalent
to $\Mod\Gamma$, where $\Gamma=H^{0}(\tilde\Gamma)=\End_{\cd(A)}(M)$. This
$t$-structure restricts to a $t$-structure on
$\cd_{fd}(\tilde{\Gamma})$, denoted by $(\cd^{\leq 0}_{fd},\cd^{\geq
0}_{fd})$, whose heart is equivalent to $\mod\Gamma$. Moreover,
there is a standard co-$t$-structure $(\cp_{\geq 0},\cp_{\leq 0})$
on $\per(\tilde\Gamma)$, see Section~\ref{a:fin-dim-negative-dg}.

The object $M$ has a natural dg $\tilde{\Gamma}$-$\Lambda$-bimodule
structure. Moreover, since it generates $K^b(\proj\Lambda)$,
it follows from~\cite[Lemma 6.1 (a)]{Keller94} that there are triangle equivalences
\[\xymatrix@R=1.3pc{F=?\lten_{\tilde\Gamma}M:\hspace{-20pt}&\mathcal{D}(\tilde\Gamma)\ar[rr]^{\sim}&&\cd(\Lambda)\ar@{=}[r]&\mathcal{D}(\Mod\Lambda)\\
&\cd_{fd}(\tilde\Gamma)\ar@{^{(}->}[u]\ar[rr]^{\sim}&&\cd_{fd}(\Lambda)\ar@{=}[r]\ar@{^{(}->}[u]&\cd^b(\mod\Lambda)\ar@{^{(}->}[u]&\\
&\per(\tilde\Gamma)\ar@{^{(}->}[u]\ar[rr]^{\sim}&&\per(\Lambda)\ar@{=}[r]\ar@{^{(}->}[u]&K^b(\proj\Lambda)\ar@{^{(}->}[u].}\]
These equivalences take $\tilde\Gamma$ to $M$. 
The following special case of 
Theorem~\ref{t:silting-have-same-number-of-summands} is a consequence.

\begin{corollary}\label{c:silting-have-fixed-number-of-summands}
The number of indecomposable direct summands of $M$ equals the rank
of the Grothendieck group of $K^b(\proj\Lambda)$. In particular,
any two basic silting objects of $K^b(\proj\Lambda)$ have the same
number of indecomposable direct summands.
\end{corollary}
\begin{proof} The number of indecomposable direct summands of $M$
equals the rank of the Grothendieck group of $\mod\Gamma$, which
equals the rank of the Grothendieck group of
$\cd_{fd}(\tilde\Gamma)\cong\cd^b(\mod\Lambda)$ since $\mod\Gamma$
is the heart of a bounded $t$-structure (Lemma~\ref{l:heart}).
\end{proof}

Write $M=M_1\oplus\ldots\oplus M_r$ with $M_i$ indecomposable.
Suppose that $X_1,\ldots,X_r$ are objects in $\cd^b(\mod\Lambda)$
such that {their endomorphism algebras $R_1,\ldots,R_r$ are division algebras and that the following formula holds} for $i,j=1,\ldots,r$ and $m\in\mathbb{Z}$
\[\Hom(M_i,\Sigma^m X_j)=\begin{cases} {{}_{R_j}R_j} & \text{ if } i=j \text{ and } m=0,\\ 0 & \text{ otherwise}.\end{cases}\]
Then up to isomorphism,  the objects
$X_1,\ldots,X_r$ are sent by the derived equivalence $?\lten_{\tilde\Gamma}M$ to a complete set of {pairwise} non-isomorphic simple
$\Gamma$-modules, see Section~\ref{ss:simpledg}.

\begin{lemma}\label{l:hom-duality}
\begin{itemize}
\item[(a)] Let $X_1',\ldots,X_r'$ be objects of $\cd^b(\mod\Lambda)$ such that {the following formula holds} for $1\leq i,j\leq r$ and $m\in\mathbb{Z}$
\[\Hom(M_i,\Sigma^m X_j')=\begin{cases} {{}_{R_j}R_j} & \text{ if } i=j \text{ and } m=0,\\ 0 & \text{ otherwise}.\end{cases}\]
Then $X_i\cong X_i'$ for any $i=1,\ldots,r$.
\item[(b)] Let $M_1',\ldots,M_r'$ be objects of $K^b(\proj\Lambda)$
such that the following formula holds for $1\leq i,j\leq r$ and $m\in\mathbb{Z}$
\[\Hom(M_i',\Sigma^m X_j)=\begin{cases} {{}_{R_j}R_j} & \text{ if } i=j \text{ and } m=0,\\ 0 & \text{ otherwise}.\end{cases}\]
Then $M_i\cong M_i'$ for any $i=1,\ldots,r$.
\end{itemize}
\end{lemma}
\begin{proof} This follows from the corresponding result in
$\cd(\tilde\Gamma)$, see Section~\ref{ss:simpledg}.
\end{proof}

\subsection{From co-$t$-structures to silting objects}\label{ss:from-co-t-str-to-silting}

Let $(\cc_{\geq 0},\cc_{\leq 0})$ be a bounded co-$t$-structure of
$K^b(\proj\Lambda)$. By Lemma~\ref{l:co-heart}, the co-heart $\ca=\cc_{\geq 0}\cap\cc_{\leq 0}$
is a silting subcategory of {$K^b(\proj\Lambda)$}. Since $\Lambda$ is a silting
object of $K^b(\proj\Lambda)$, it follows from
Theorem~\ref{t:silting-have-same-number-of-summands} that $\ca$ has
an additive generator, say $M$, \ie $\ca=\add(M)$. Then $M$ is a
silting object in $K^b(\proj\Lambda)$. Define \begin{eqnarray*}
\phi_{14}(\cc_{\geq 0},\cc_{\leq 0})&=&M.\end{eqnarray*}

\subsection{From $t$-structures to simple-minded collections}\label{ss:from-t-str-to-smc} Let
$(\cc^{\leq 0},\cc^{\geq 0})$ be a bounded $t$-structure of
$\cd^b(\mod\Lambda)$ with length heart $\ca$. Boundedness
implies that the Grothendieck group of $\ca$ is isomorphic
to the Grothendieck group of $\cd^b(\mod\Lambda)$, which is free, say, of
rank $r$. Therefore, $\ca$ has precisely $r$ isomorphism classes of
simple objects, say $X_1,\ldots,X_r$. By Lemma~\ref{l:heart} (f),
$X_1,\ldots,X_r$ is a simple-minded collection in
$\cd^b(\mod\Lambda)$. Define \begin{eqnarray*}\phi_{23}(\cc^{\leq
0},\cc^{\geq 0})&=&\{X_1,\ldots,X_r\}.\end{eqnarray*}

\subsection{From silting objects to simple-minded collections, $t$-structures and
co-$t$-structures}\label{s:from-silting}

Let $M$ be a silting object of $K^b(\proj\Lambda)$. Define full
subcategories of $\cc$
\begin{eqnarray*}
\cc^{\leq 0}&=& \{N\in\cd^b(\mod\Lambda)\mid\Hom(M,\Sigma^m N)=0,
~\forall~
m>0\},\\
\cc^{\geq 0}&=& \{N\in\cd^b(\mod\Lambda)\mid\Hom(M,\Sigma^m N)=0,
~\forall~
m<0\},\\
\cc_{\leq 0}&=& \text{the additive closure of the extension closure of $\Sigma^m M$, $m\geq 0$ in } K^b(\proj \Lambda), \\
\cc_{\geq 0}&=& \text{the additive closure of the extension closure of
$\Sigma^m M$, $m\leq 0$ in } K^b(\proj \Lambda).
\end{eqnarray*}

\begin{lemma}\label{l:from-silting}
\begin{itemize}
\item[(a)] The pair $(\cc^{\leq 0},\cc^{\geq 0})$ is a bounded $t$-structure
on $\cd^b(\mod\Lambda)$ whose heart is equivalent to $\mod\Gamma$
for $\Gamma=\End(M)$. Write $M=M_1\oplus\ldots\oplus M_r$ and let
$X_1,\ldots,X_r$ be the corresponding simple objects of the heart with endomorphism algebras $R_1,\ldots,R_r$ respectively.
Then the following formula holds for $1\leq i,j\leq r$ and $m\in\mathbb{Z}$
\[\Hom(M_i,\Sigma^m X_j)=\begin{cases} {{}_{R_j}R_j} & \text{ if } i=j \text{ and } m=0,\\ 0 & \text{ otherwise}.\end{cases}\]
\item[(b)] The pair $(\cc_{\geq 0},\cc_{\leq 0})$ is a bounded
co-$t$-structure on $K^b(\proj\Lambda)$ whose co-heart is $\add(M)$.
\end{itemize}
\end{lemma}

The first statement of part (a) is proved by Keller and Vossieck~\cite{KellerVossieck88} in the case when $\Lambda$
is the path algebra of a Dynkin quiver and by Assem, Souto and Trepode~\cite{AssemSoutoTrepode08}
 in the case when $\Lambda$ is hereditary.

\begin{proof}
Let $\tilde\Gamma$ be the truncated dg endomorphism algebra of $M$,
see Section~\ref{ss:silting-der-equiv}. Then $\per(\tilde\Gamma)$
has a standard bounded co-$t$-structure $(\cp_{\geq 0},\cp_{\leq
0})$ and $\cd_{fd}(\tilde\Gamma)$ has a standard bounded
$t$-structure $(\cd_{fd}^{\leq 0},\cd_{fd}^{\geq 0})$ with heart
equivalent to $\mod\Gamma$. One checks that the triangle equivalence
$?\lten_{\tilde\Gamma}M$ takes $(\cp_{\geq 0},\cp_{\leq 0})$ to
$(\cc_{\geq 0},\cc_{\leq 0})$ and it takes $(\cd_{fd}^{\leq
0},\cd_{fd}^{\geq 0})$ to $(\cc^{\leq 0},\cc^{\geq 0})$.
\end{proof}

Define \begin{eqnarray*}
\phi_{31}(M)&=&(\cc^{\leq 0},\cc^{\geq
0}),\\
\phi_{41}(M)&=&(\cc_{\geq 0},\cc_{\leq 0}),\\
\phi_{21}(M)&=&\{X_1,\ldots,X_r\}.
\end{eqnarray*}

\subsection{From simple-minded collections to $t$-structures}
\label{ss:from-smc-to-t-str}

Let $X_1,\ldots,X_r$ be a simple-minded collection of
$\cd^b(\mod\Lambda)$. Let $\cc^{\leq 0}$ (respectively, $\cc^{\geq
0}$) be  the extension closure of $\{\Sigma^m X_i\mid i=1,\ldots,r,
m\geq 0\}$ (respectively, $\{\Sigma^m X_i\mid i=1,\ldots,r, m\leq 0\}$)
in $\cd^b(\mod\Lambda)$.

\begin{proposition}\label{p:simpleminded-to-t-str}
 The pair $(\cc^{\leq 0},\cc^{\geq 0})$ is
 a bounded
$t$-structure on $\cd^b(\mod\Lambda)$. Moreover, the heart of this
$t$-structure is a length category with simple objects
$X_1,\ldots,X_r$. The same results hold true with $\cd^b(\mod\Lambda)$ replaced by a Hom-finite
Krull--Schmidt triangulated category $\cc$.
\end{proposition}
\begin{proof}  The first two statements are~\cite[Corollary 3 and Proposition 4]{Al-Nofayee09}.
The proof there still works if we replace $\cd^b(\mod\Lambda)$ by $\cc$.
\end{proof}

Define \begin{eqnarray*} \phi_{32}(X_1,\ldots,X_r)&=&(\cc^{\leq
0},\cc^{\geq 0}).
\end{eqnarray*}
Later we will show that the heart of this $t$-structure {always}
is equivalent to the category of finite-dimensional modules over a
finite-dimensional algebra (Corollary~\ref{c:length-heart}). This
was proved by Al-Nofayee for self-injective algebras $\Lambda$,
see~\cite[Theorem 7]{Al-Nofayee09}.

\begin{corollary}
Any two simple-minded collections in $\cd^b(\mod\Lambda)$ have the
same cardinality.
\end{corollary}
\begin{proof}
By Proposition~\ref{p:simpleminded-to-t-str}, the cardinality of a
simple-minded collection equals the rank of the Grothendieck group
of $\cd^b(\mod\Lambda)$. The assertion follows.
\end{proof}

\subsection{From simple-minded collections to silting
objects}\label{s:from-simple-minded-to-silting}
 Let $X_1,\ldots,X_r$ be a
simple-minded collection in $\cd^b(\mod\Lambda)$. We will construct
a silting object $\nu^{-1}T$ of $K^b(\proj\Lambda)$ following a method of Rickard~\cite{Rickard02}. 
Then we define
\begin{eqnarray*}\phi_{12}(X_1,\ldots,X_r)&=&\nu^{-1}T.\end{eqnarray*}
The
same construction is studied by
Keller and Nicol\'as~\cite{KellerNicolas12} in the context of positive
dg algebras.
In the case of $\Lambda$ being hereditary, 
Buan, Reiten and Thomas~\cite{BuanReitenThomas12} give an elegant construction 
of $\nu^{-1}(T)$ using the Braid group action on exceptional
sequences. Unfortunately, their construction cannot be generalised.

\smallskip

{Let $R_1,\ldots,R_r$ be the endomorphism algebras of 
$X_1,\ldots,X_r$, respectively.}

Set $X_i^{(0)}=X_i$. Suppose $X_i^{(n-1)}$ is constructed. For
$i,j=1,\ldots,r$ and $m<0$, let $B(j,m,i)$ be a basis of
$\Hom(\Sigma^m X_j,X_i^{(n-1)})$ {over $R_j$}. Put
\[Z_i^{(n-1)}=\bigoplus_{m<0}\bigoplus_j \bigoplus_{B(j,m,i)}\Sigma^m
X_{j}\] and let $\alpha_i^{(n-1)}:Z_i^{(n-1)}\rightarrow
X_i^{(n-1)}$ be the map whose component corresponding to $f\in
B(j,m,i)$ is exactly $f$.

Let $X_i^{(n)}$ be a cone of $\alpha_i^{(n-1)}$ and form the
corresponding triangle
\[\xymatrix{Z_i^{(n-1)}\ar[r]^{\alpha_i^{(n-1)}}
& X_i^{(n-1)}\ar[r]^{\beta_i^{(n-1)}}& X_i^{(n)}\ar[r]& \Sigma
Z_i^{(n-1)}.}\] Inductively, a sequence of morphisms in
$\cd(\Mod\Lambda)$ is constructed:
\[\xymatrix{X_i^{(0)}\ar[r]^{\beta_i^{(0)}} & X_i^{(1)}\ar[r] &\ldots\ar[r] &
X_i^{(n-1)}\ar[r]^{\beta_i^{(n-1)}} & X_i^{(n)}\ar[r] & \ldots.}\]
Let $T_i$ be the homotopy colimit of this sequence. That is, up to isomorphism,
$T_i$ is defined by the following triangle
\[\xymatrix{\bigoplus_{n\geq 0}X_i^{(n)} \ar[r]^{id-\beta} & \bigoplus_{n\geq 0}X_i^{(n)}\ar[r] & T_i\ar[r] &\Sigma \bigoplus_{n\geq 0}X_i^{(n)}}.\]
Here $\beta=(\beta_{mn})$ is the square matrix with rows and columns labeled by non-negative integers and
with entries $\beta_{mn}=\beta_i^{(n)}$ if $n+1=m$ and $0$ otherwise.

These properties of $T_i$'s were proved by Rickard in~\cite{Rickard02}
for symmetric algebras $\Lambda$ over algebraically closed fields. Rickard remarked that they hold for arbitrary fields, see~\cite[Section 8]{Rickard02}. In fact, his proofs {verbatim} carry over  to general finite-dimensional algebras.

\begin{lemma}\label{L:T}
\begin{itemize}
\item[(a)] \emph{(\cite[Lemma 5.4]{Rickard02})} For $1\leq i,j\leq r$, and
$m\in\mathbb{Z}$,
\[\Hom(X_j,\Sigma^m T_{i})=\begin{cases} {(R_j)_{R_j}} & \text{if }i=j\text{ and } m=0,\\
                                           0 & \text{otherwise}.
                                           \end{cases}\]

\item[(b)] \emph{(\cite[Lemma 5.5]{Rickard02})} For each $1\leq i\leq r$, $T_i$ is
quasi-isomorphic to a bounded complex of finitely generated injective $\Lambda$-modules.

\item[(c)] \emph{(\cite[Lemma 5.8]{Rickard02})} Let $C$ be an object of
$\mathcal{D}^{-}(\mod \Lambda)$. If $\Hom(C,\Sigma^m T_i)=0$ for all
$m\in\mathbb{Z}$ and all $1\leq i\leq r$, then $C=0$.
\end{itemize}
\end{lemma}

From now on we assume that $T_i$ is a bounded complex of finitely generated injective $\Lambda$-modules.
Recall from Section~\ref{ss:nakayama-for-fd-alg} that the Nakayama
functor $\nu$ and  the inverse Nakayama functor  $\nu^{-1}$ are
quasi-inverse triangle equivalences between $K^b(\proj\Lambda)$ and
$K^b(\inj\Lambda)$ The following is a consequence of Lemma~\ref{L:T}
and the Auslander--Reiten formula.

\begin{lemma}\label{L:nuinverseT}
\begin{itemize}
\item[(a)] For $1\leq i,j\leq r$, and $m\in\mathbb{Z}$,
\[\Hom(\nu^{-1}T_{i},\Sigma^m X_j)=\begin{cases} {}_{R_j}R_j & \text{if }i=j\text{ and } m=0,\\
                                           0 & \text{otherwise}.
                                           \end{cases}\]
\item[(b)] For each $1\leq i\leq r$, $\nu^{-1}T_i$ is a
bounded complex of finitely generated projective $\Lambda$-modules.
\item[(c)] Let $C$
be an object of $\mathcal{D}^{-}(\mod \Lambda)$. If
$\Hom(\nu^{-1}T_i,\Sigma^m C)=0$ for all $m\in\mathbb{Z}$ and all
$1\leq i\leq r$, then $C=0$.
\end{itemize}
\end{lemma}

Put $T=\bigoplus_{i=1}^{r} T_i$ and
$\nu^{-1}T=\bigoplus_{i=1}^{r}\nu^{-1}T_i$.

\begin{lemma}\label{L:vanishingofposext}
We have $\Hom(\nu^{-1}T,\Sigma^m T)=0$ for $m<0$. Equivalently,
$\Hom(\nu^{-1}T,\Sigma^m \nu^{-1}T)=\Hom(T,\Sigma^m T)=0$ for $m>0$.
\end{lemma}
\begin{proof} Same as the proof of~\cite[Lemma 5.7]{Rickard02}, with the
$T_i$ in the first entry of $\Hom$ there replaced by $\nu^{-1}T_i$.
\end{proof}

It follows from Lemma~\ref{L:nuinverseT} (c) that $\nu^{-1}T$
generates $K^b(\proj\Lambda)$. Combining this with
Lemma~\ref{L:vanishingofposext} implies 
%that $\nu^{-1}T$ is a
%silting object of $K^b(\proj\Lambda)$. { Namely,
\begin{proposition}
$\nu^{-1}T$ is a silting object of $K^b(\proj\Lambda)$.
\end{proposition}

\medskip

Rickard's construction was originally motivated by constructing tilting complexes over symmetric algebras which yield certain derived
equivalences, {see \cite[Theorem 5.1]{Rickard02}. His work was later generalised by Al-Nofayee to self-injective algebras, see \cite[Theorem 4]{Al-Nofayee07}.}

\subsection{From co-$t$-structures to
$t$-structures}\label{ss:from-co-t-to-t}

Let $(\cc_{\geq 0},\cc_{\leq 0})$ be a bounded co-$t$-structure of
$K^b(\proj\Lambda)$. Let \begin{eqnarray*} \cc^{\leq 0} &=&
\{N\in\cd^b(\mod\Lambda)\mid \Hom(M,N)=0,~\forall~M\in\Sigma^{-1}\cc_{\geq
0}\}\\
\cc^{\geq 0} &=&
\{N\in\cd^b(\mod\Lambda)\mid \Hom(M,N)=0,~\forall~M\in\Sigma^{-1}\cc_{\leq
0}\}.
\end{eqnarray*}

\begin{lemma}
The pair $(\cc^{\leq 0},\cc^{\geq 0})$ is a bounded $t$-structure on
$\cd^b(\mod\Lambda)$ with length heart.
\end{lemma}
\begin{proof}
Because $(\cc^{\leq 0},\cc^{\geq
0})=\phi_{31}\circ\phi_{14}(\cc_{\geq 0},\cc_{\leq 0})$.
\end{proof}
By definition $(\cc^{\leq 0},\cc^{\geq 0})$ is right orthogonal to the given co-$t$-structure in the sense of Bondarko~\cite{Bondarko10a}.
Define \begin{eqnarray*} \phi_{34}(\cc_{\geq 0},\cc_{\leq
0})&=&(\cc^{\leq 0},\cc^{\geq 0}).
\end{eqnarray*}

If $\Lambda$ has finite global dimension, then
$K^b(\proj\Lambda)$ is identified with $\cd^b(\mod\Lambda)$. As a
consequence, $\cc_{\leq 0}=\cc^{\leq 0}$ and $\cc^{\geq 0}=\nu\cc_{\geq 0}$. Thus the $t$-structure
$(\cc^{\leq 0},\cc^{\geq 0})$ is right adjacent to the given
co-$t$-structure $(\cc_{\geq 0},\cc_{\leq 0})$ in the sense of
Bondarko~\cite{Bondarko10}.

\subsection{Some remarks}\label{ss:generality-of-the-maps}
Some of the maps $\phi_{ij}$ are defined in more general setups:

\begin{itemize}
\item[--] $\phi_{14}$ and
$\phi_{41}$ are defined for all triangulated categories, with
silting objects replaced by silting subcategories, by
Proposition~\ref{p:AI's-from-silting-to-co-t-str} and Lemma~\ref{l:co-heart},
see also~\cite{Bondarko10,KellerNicolas11,MendozaSaenzSantiagoSouto10}.
\item[--] $\phi_{23}$ is defined for all triangulated
categories, with simple-minded collections allowed to contain
infinitely many objects (Lemma~\ref{l:heart}).
\item[--] $\phi_{32}$ is defined for all algebraic
triangulated categories (see~\cite{KellerNicolas12}) and for Hom-finite
Krull--Schmidt triangulated categories
(see Proposition~\ref{p:simpleminded-to-t-str}).
\item[--] $\phi_{21}$ and $\phi_{31}$ are defined for all
algebraic triangulated categories (replacing $K^b(\proj\Lambda)$),
with $\cd^b(\mod\Lambda)$ replaced by a suitable triangulated
category; then we may follow the arguments in
Sections~\ref{s:standard-t-structure} and~\ref{s:from-silting}.
\item[--] $\phi_{34}$ is defined for all algebraic triangulated
categories (replacing $K^b(\proj\Lambda)$), with
$\cd^b(\mod\Lambda)$ replaced by a suitable triangulated category.
Then we may follow the argument in Section~\ref{ss:from-co-t-to-t}.
\item[--] $\phi_{12}$ is defined for {finite-dimensional non-positive dg algebras, since these dg algebras behave like finite-dimensional algebras from the perspective of derived categories. Similarly, $\phi_{12}$ is defined for  homologically
smooth non-positive dg algebras,} see~\cite{KellerNicolas11}.
\end{itemize}

\section{The correspondences are bijections}\label{s:bijectivity} \label{section-bijections}

%Let $\Lambda$ be a finite-dimensional algebra over $K$. In
%Section~\ref{s:the-maps} we defined the following maps

Let $\Lambda$ be a finite-dimensional $K$-algebra.
In the preceding section we defined the maps $\phi_{ij}$. In this section we will show that they are bijections. See \cite{AssemSoutoTrepode08,Vitoria12} for related work, focussing on
piecewise hereditary algebras.

\begin{theorem}\label{t:main-thm-bijections} \label{maintheorembijections}
The $\phi_{ij}$'s defined in Section~\ref{section-maps} are bijective.
In particular, there are one-to-one correspondences between
\begin{itemize}
 \item[(1)] equivalence classes of silting objects in $K^b(\proj \Lambda)$,
 \item[(2)] equivalence classes of simple-minded collections in $\cd^b(\mod \Lambda)$,
 \item[(3)] bounded $t$-structures on $\cd^b(\mod \Lambda)$ with length heart,
 \item[(4)] bounded co-$t$-structures on $K^b(\proj \Lambda)$.
\end{itemize}

\end{theorem}

There is an immediate consequence:
\begin{corollary}\label{c:length-heart}
Let $\ca$ be the heart of a bounded $t$-structure on
$\cd^b(\mod\Lambda)$. If $\ca$ is a length category, then $\ca$ is
equivalent to $\mod\Gamma$ for some finite-dimensional algebra
$\Gamma$. \end{corollary}
\begin{proof}{
By Theorem~\ref{maintheorembijections}, such a $t$-structure is of the form 
$\phi_{31}(M)$ for some silting object $M$ of $K^b(\proj\Lambda)$. The 
result then follows from Lemma~\ref{l:from-silting} (a).}
\end{proof}

The proof of the theorem is divided into several lemmas, which are
consequences of the material collected in the previous sections.

\begin{lemma}\label{l:silting-and-co-t-str}
 The maps $\phi_{14}$ and $\phi_{41}$ are inverse to each other.
\end{lemma}

\begin{proof} Let $M$ be a basic silting object. The definitions of
$\phi_{14}$ and $\phi_{41}$ and Lemma~\ref{l:from-silting} (b) imply that
$\phi_{14}\circ\phi_{41}(M)\cong M$.

Let $(\cc_{\geq 0},\cc_{\leq 0})$ be a bounded co-$t$-structure on
$K^b(\proj\Lambda)$. It follows from Lemma~\ref{l:co-heart} that
$\phi_{41}\circ\phi_{14}(\cc_{\geq 0},\cc_{\leq 0})=(\cc_{\geq
0},\cc_{\leq 0})$.
\end{proof}

Recall from Section~\ref{ss:generality-of-the-maps} that $\phi_{14}$ and
$\phi_{41}$ are defined in full generality. Lemma~\ref{l:silting-and-co-t-str}
holds in full generality as well,
see~\cite[Corollary 5.8]{MendozaSaenzSantiagoSouto10} and \cite{KellerNicolas11}.

\begin{lemma}
 The maps $\phi_{21}$ and $\phi_{12}$ are inverse to each other.
\end{lemma}
\begin{proof}
 This follows from the Hom-duality:
Lemma~\ref{L:nuinverseT} (a), Lemma~\ref{l:from-silting} (a) and
Lemma~\ref{l:hom-duality}.
\end{proof}

\begin{lemma}
 The maps $\phi_{23}$ and $\phi_{32}$ are inverse to each other.
\end{lemma}

\begin{proof} Let $X_1,\ldots,X_r$ be a simple-minded collection in
$\cd^b(\mod\Lambda)$. It follows from
Proposition~\ref{p:simpleminded-to-t-str} that
$\phi_{23}\circ\phi_{32}(X_1,\ldots,X_r)=\{X_1,\ldots,X_r\}$.

Let $(\cc^{\leq 0},\cc^{\geq 0})$ be a bounded $t$-structure on
$\cd^b(\mod\Lambda)$ with length heart. It follows from
Lemma~\ref{l:heart} that $\phi_{32}\circ\phi_{23}(\cc^{\leq
0},\cc^{\geq 0})=(\cc^{\leq 0},\cc^{\geq 0})$.
\end{proof}

\begin{lemma}\label{l:transitivity} For a triple $i,j,k$ such that
$\phi_{ij}$, $\phi_{jk}$ and
$\phi_{ik}$ are defined, there is the equality $\phi_{ij}\circ\phi_{jk}=\phi_{ik}$. In
particular, $\phi_{31}$ and $\phi_{34}$ are bijective.
\end{lemma}
\begin{proof} In view of the preceding three lemmas, it suffices to
prove $\phi_{23}\circ\phi_{31}=\phi_{21}$ and
$\phi_{31}\circ\phi_{14}=\phi_{34}$, which is clear from the
definitions.
\end{proof}

\section{Mutations and partial orders}\label{s:mutation} \label{section-mutation}

In this section we introduce mutations and partial orders on the four concepts
in Section~\ref{section-fourconcepts}, and we show that
the maps defined in Section~\ref{section-maps} commute with mutations and
preserve the partial orders.

Let $\cc$ be a Hom-finite Krull--Schmidt triangulated category with suspension functor $\Sigma$.

\subsection{Silting objects}\label{ss:silting-obj-mutation}

We follow~\cite{AiharaIyama12,BuanReitenThomas11} to define silting mutation.
Let $M$ be a silting object in $\cc$. We assume that $M$ is basic and $M=M_1\oplus\ldots\oplus M_r$ is a
decomposition into indecomposable objects. Let $i=1,\ldots,r$.
The \emph{left mutation} of $M$ at the direct summand $M_i$ is the object
$\mu^+_i(M)=M_i'\oplus\bigoplus_{j\neq i}M_j$ where $M_i'$ is the cone of
the minimal left $\add(\bigoplus_{j\neq i}M_j)$-approximation of $M_i$
$$\xymatrix{M_i\ar[r]& E.}$$
Similarly one can define the \emph{right mutation} $\mu^-_i(M)$.

\begin{theorem}\label{t:mutation-silting}
 \emph{(\cite[Theorem {2.31} and Proposition {2.33}]{AiharaIyama12})} The objects $\mu^+_i(M)$ and $\mu^-_i(M)$ are silting objects.
Moreover, $\mu^{+}_i\circ\mu^-_i(M)\cong M\cong\mu^{-}_i\circ\mu^+_i(M)$.
\end{theorem}
\newcommand{\silt}{\opname{silt}}

Let $\silt\cc$ be the set of isomorphism classes of
basic tilting objects of $\cc$. The \emph{silting quiver} of $\cc$ has the
elements in $\silt\cc$ as vertices.
For $P,P'\in\silt\cc$, there are arrows from $P$ to $P'$ if and only if
$P'$ is obtained from $P$ by a left mutation, in which case there is
precisely one arrow. See~\cite[Section 2.6]{AiharaIyama12}.

For $P,P'\in\silt\cc$, define $P\geq P'$ if $\Hom(P,\Sigma^m P')=0$ for any $m>0$.
According to~\cite[Theorem {2.11}]{AiharaIyama12}, $\geq$ is a partial order on $\silt\cc$.

\begin{theorem}
\emph{(\cite[Theorem {2.35}]{AiharaIyama12})} The Hasse diagram of $(\silt\cc,\geq)$ is the silting quiver of $\cc$.
\end{theorem}

Next we define (a generalisation of)
the Brenner--Butler tilting module for a
finite-dimensional algebra, and show that it is a
left mutation of the free module of rank $1$. The corresponding
right mutation is the Okuyama--Rickard complex,
see~\cite[Section 2.7]{AiharaIyama12}. Let $\Lambda$ be a
finite-dimensional basic algebra and $1=e_1+\ldots+e_n$ be a
decomposition of the unity into the sum of primitive idempotents
and $\Lambda=P_1\oplus\ldots\oplus P_n$ the corresponding decomposition of the free module of rank $1$.
Fix $i=1,\ldots,n$ and let $S_i$ be the corresponding simple module and let $S^+_i=\mathrm{D}(\Lambda/\Lambda
(1-e_i)\Lambda)$. Assume that
\begin{itemize}
 \item[$\cdot$] $S^+_i$ is not injective,
 \item[$\cdot$] the projective dimension of $\tau^{-1}_{\mod\Lambda}S^+_i$ is at most $1$.
\end{itemize}
\begin{definition}Define the \emph{BB tilting module} with respect to $i$ by
$$T=\tau^{-1}_{\mod\Lambda}S^+_i\oplus\bigoplus_{j\neq i}P_j.$$
We call it the \emph{APR tilting module} if
$\Lambda/\Lambda(1-e_i)\Lambda$ is projective as a $\Lambda$-module.
\end{definition}
When $\Lambda/\Lambda(1-e_i)\Lambda$ is {a division algebra} {(\ie there are no loops in the quiver of $\Lambda$ at the vertex $i$)}, this
specialises to the `classical' BB tilting
module~\cite{BrennerButler80} and APR tilting
module~\cite{AuslanderPlatzeckReiten79}. The following proposition generalises~\cite[Theorem 2.53]{AiharaIyama12}.

\begin{proposition}\label{p:bb-tilting}
 %The following hold
\begin{itemize}
 \item[(a)] $T$ is isomorphic to the left mutation $\mu^+_i(\Lambda)$ of $\Lambda$.
 \item[(b)] $T$ is a tilting $\Lambda$-module of projective dimension at most $1$.
\end{itemize}
\end{proposition}

\begin{proof} {We modify the proof in}~\cite{AiharaIyama12}.
Take a minimal injective copresentation of $S^+_i$:
\[\xymatrix{0\ar[r] & S^+_i\ar[r] & \mathrm{D}(e_i\Lambda)\ar[r]^(0.6)f&I.}\]
Since
$\Ext^1_\Lambda(S_i,S^+_i)=\Ext^1_{\Lambda/\Lambda(1-e_i)\Lambda}(S_i,S^+_i)=0$,
it follows that the injective module $I$ belongs to
\mbox{$\add\mathrm{D}((1-e_i)\Lambda)$.} Applying the inverse Nakayama
functor $\nu^{-1}_{\mod\Lambda}$ yields an exact sequence
\[\xymatrix{P_i\ar[r]^(0.4){\nu^{-1}_{\mod\Lambda}f}&\nu^{-1}_{\mod\Lambda}I\ar[r]&\tau^{-1}_{\mod\Lambda}S^+_i\ar[r] & 0.}\]
Moreover, $\nu^{-1}_{\mod\Lambda}f$ is a minimal left approximation
of $P_i$ in $\add(P_j,j\neq i)$. Since the projective dimension of
$\tau^{-1}_{\mod\Lambda}S^+_i$ is at most $1$, it follows that
$\nu^{-1}_{\mod\Lambda}f$ is injective. This completes the proof for
(a).

(b) follows from~\cite[Theorem {2.32}]{AiharaIyama12}.
\end{proof}

\subsection{Simple-minded collections}\label{ss:simple-minded-collection-mutation}

Let $X_1,\ldots,X_r$ be a simple-minded collection in $\cc$ and fix
$i=1,\ldots,r$. Let $\cx_i$ denote the extension closure of $X_i$ in
$\cc$. Assume that for any $j$ the object $\Sigma^{-1}X_j$ admits a
minimal left approximation $g_j:\Sigma^{-1}X_j\rightarrow X_{ij}$ in
$\cx_i$. \begin{definition}The \emph{left mutation}
$\mu^+_i(X_1,\ldots,X_r)$ of $X_1,\ldots,X_r$ at $X_i$ is a new
collection $X_1',\ldots,X_r'$ such that $X_i'=\Sigma X_i$ and $X_j'$
($j\neq i$) is the cone of the above left approximation
\[\xymatrix{\Sigma^{-1}X_j\ar[r]^{g_j}&
X_{ij}.}\] Similarly one defines the \emph{right mutation}
$\mu^-_i(X_1,\ldots,X_r)$. \end{definition}
This generalises
Kontsevich--Soibelman's mutation of spherical
collections~\cite[Section 8.1]{KontsevichSoibelman08} {and appeared in \cite{KingQiu11} in the case of derived categories of acyclic quivers}.

\begin{proposition}\label{p:mutation-simple-minded}
\begin{itemize}
\item[(a)] $\mu^{+}_i\circ\mu^-_i(X_1,\ldots,X_r)\cong(X_1,\ldots,X_r)\cong \mu^{-}_i\circ\mu^+_i(X_1,\ldots,X_r)$.
\item[(b)]
Assume that
\begin{itemize}
\item[$\cdot$] for any $j\neq i$ the object $\Sigma^{-1}X_j$ admits a
minimal left approximation $g_j:\Sigma^{-1}X_j\rightarrow X_{ij}$ in
$\cx_i$,
\item[$\cdot$] the induced map
$\Hom(g_j,X_i):\Hom(X_{ij},X_i)\rightarrow\Hom(\Sigma^{-1}X_j,X_i)$
is injective,
\item[$\cdot$] the induced map
$\Hom(g_j,\Sigma X_i):\Hom(X_{ij},\Sigma
X_i)\rightarrow\Hom(\Sigma^{-1}X_j,\Sigma X_i)$ is injective.
\end{itemize}
Then the collection $\mu^+_i(X_1,\ldots,X_r)$ is simple-minded.
\item[(c)]
Assume that
\begin{itemize}
\item[$\cdot$] for any $j\neq i$ the object $X_j$ admits a
minimal right approximation $g_j^-:\Sigma^{-1}X_{ij}^-\rightarrow
X_{j}$ in $\Sigma^{-1}\cx_i$,
\item[$\cdot$] the induced map
$\Hom(X_i,\Sigma g_j^-):\Hom(X_{i},
X_{ij}^-)\rightarrow\Hom(X_i,\Sigma X_j)$ is injective,
\item[$\cdot$] the induced map
$\Hom(X_i,\Sigma^2 g_j^-,):\Hom(X_i,\Sigma
X_{ij}^-)\rightarrow\Hom(X_i,\Sigma^2 X_j)$ is injective.
\end{itemize}
Then the collection $\mu^-_i(X_1,\ldots,X_r)$ is simple-minded.
\end{itemize}
\end{proposition}
\begin{proof}
(a) Because in the triangle
\[\xymatrix{\Sigma^{-1}X_j\ar[r]^{g_j}&X_{ij}\ar[r]^{g_j^-}&X_j'\ar[r]&X_j}\]
$g_j$ is a minimal left approximation of $\Sigma^{-1}X_j$ in $\cx_i$
if and only if $g_j^-$ is a minimal right approximation of $X_j$ in
$\cx_i=\Sigma^{-1}(\Sigma \cx_i)$.

(b) and (c) The proof uses long exact Hom sequences induced from the
defining triangles of the $X_j'$. We leave it to the reader.
\end{proof}

\begin{remark}
In the course of the proof of Proposition~\ref{p:mutation-simple-minded} (b) and (c), one notices that the collection of endomorphism algebras of the mutated simple-minded collection is the same as that of the given simple-minded collection.
\end{remark}

If $\Hom(X_i,\Sigma X_i)=0$, then $\cx_i=\add (X_i)$. In this case,
all six assumptions in Proposition~\ref{p:mutation-simple-minded} (b) and (c) are satisfied.
% and the mutations were studied in~\cite[Section 8.1]{KontsevichSoibelman08} for $\cc$
%being $3$-Calabi--Yau and in~\cite[Section 5.1]{Qiu11} for $\cc$ being the bounded derived category of a Dynkin quiver.

\begin{lemma} Let $\Lambda$ be a finite-dimensional algebra and let $X_1,\ldots,X_r$ be a simple-minded collection in $\cd^b(\mod\Lambda)$. Let $i=1,\ldots,r$.
Then the left mutation $\mu^+_i(X_1,\ldots,X_r)$ and the right
mutation $\mu^-_i(X_1,\ldots,X_r)$ are again simple-minded
collections.
\end{lemma}
\begin{proof}
 We will show that the three assumptions in Proposition~\ref{p:mutation-simple-minded} (b) are satisfied,
so the left-mutated collection $\mu^+_i(X_1,\ldots,X_r)$ is a simple-minded collection. The case
for $\mu^-_i(X_1,\ldots,X_r)$ is similar.

By Proposition~\ref{p:simpleminded-to-t-str}, $X_1,\ldots,X_r$ are
the simple objects in the heart of a bounded $t$-structure on
$\cd^b(\mod\Lambda)$. Moreover, by Corollary~\ref{c:length-heart},
the heart is equivalent to $\mod\Gamma$ for some finite-dimensional
algebra $\Gamma$. We identify $\mod\Gamma$ with the heart via this
equivalence. In this way we consider $X_1,\ldots,X_r$ as simple
$\Gamma$-modules.

\newcommand{\real}{\opname{real}}
By~\cite[Section 3.1]{BeilinsonBernsteinDeligne82}, there is a
triangle functor
\[\real:\cd^b(\mod\Gamma)\rightarrow\cd^b(\mod\Lambda)\]
such that
\begin{itemize}
 \item[--] restricted to $\mod\Gamma$, $\real$ is the identity;
 \item[--] for $M,N\in\mod\Gamma$, the induced map $$\Ext^1_\Gamma(M,N)=\Hom_{\cd^b(\mod\Gamma)}(M,\Sigma N)\rightarrow\Hom_{\cd^b(\mod\Lambda)}(M,\Sigma N)$$
is bijective;
 \item[--] for $M,N\in\mod\Gamma$, the induced map $$\Ext^2_\Gamma(M,N)=\Hom_{\cd^b(\mod\Gamma)}(M,\Sigma^2 N)\rightarrow\Hom_{\cd^b(\mod\Lambda)}(M,\Sigma^2 N)$$
is injective.
\end{itemize}

For $j=1,\ldots,r$, there is a short exact sequence
\[\xymatrix{0\ar[r] &\Omega X_j\ar[r] & P_j \ar[r] & X_j\ar[r] & 0,}\]
where $P_j$ is the projective cover of $X_j$ and $\Omega X_j$ is the
first syzygy of $X_j$. Let $\cx_i$ be the extension closure of $X_i$
in $\mod\Gamma$ (by the second property of $\real$
listed in the preceding paragraph, this
is the same as the extension closure of $X_i$ in
$\cd^b(\mod\Lambda)$) and let $X_{ij}$ denote the maximal quotient
of $\Omega X_j$ belonging to $\cx_i$. There is the following push-out
diagram
\[\xymatrix{&0\ar[r] &\Omega X_j\ar[d]\ar[r] & P_j\ar[d] \ar[r] & X_j\ar[r] \ar@{=}[d]& 0\\
\xi:\hspace{-20pt}&0\ar[r] & X_{ij}\ar[r]&X_j'\ar[r] & X_j\ar[r] &
0}\]

(a) Suppose we are given an object $Y$ of $\cx_i$ and a short exact
sequence
\[\xymatrix{\eta:\hspace{-20pt}&0\ar[r]&Y\ar[r] &Z\ar[r] & X_j\ar[r] & 0.}\]
Then there is a commutative diagram
\[\xymatrix{&0\ar[r] &\Omega X_j\ar[d]\ar[r] & P_j\ar[d] \ar[r] & X_j\ar[r] \ar@{=}[d]& 0\\
\eta:\hspace{-20pt}&0\ar[r]&Y\ar[r]&Z \ar[r] & X_j\ar[r] & 0.}\]
Because $X_{ij}$ is the maximal quotient of $\Omega X_j$ belonging
to $\cx_i$, this morphism of short exact sequences factors through
$\xi$. In other words, the morphism $g_j:X_j\rightarrow\Sigma
X_{ij}$ corresponding to $\xi$ is a minimal left
$\Sigma\cx_i$-approximation.

(b) The dimension of the space $\Hom_{\cd^b(\mod\Lambda)}(\Sigma
X_{ij},\Sigma X_i)\cong\Hom_{\Gamma}(X_{ij},X_i)$ {over $\End(X_i)$} equals the number
of indecomposable direct summands of $\mathrm{top}(X_{ij})$, which
clearly equals the dimension of
$\Ext^1_{\Gamma}(X_j,X_i)\cong\Hom_{\cd^b(\mod\Lambda)}(X_j,\Sigma
X_i)$ {over $\End(X_i)$}. Therefore the induced map
\[\Hom(g_j,\Sigma X_i):\Hom_{\cd^b(\mod\Lambda)}(\Sigma X_{ij},\Sigma X_i)\longrightarrow \Hom_{\cd^b(\mod\Lambda)}(X_j,\Sigma X_i)\]
is injective since by (a) it is surjective.

(c) First observe that the following diagram is commutative
\[\xymatrix{
\Hom_{\cd^b(\mod\Lambda)}(\Sigma X_{ij},\Sigma^2 X_i)\ar[rr]^{\Hom_{\Lambda}(g_j,\Sigma^2 X_i)}&&\Hom_{\cd^b(\mod\Lambda)}(X_j,\Sigma^2 X_i)\\
\Hom_{\cd^b(\mod\Gamma)}(\Sigma X_{ij},\Sigma^2
X_i)\ar[u]^{\real}\ar[rr]^{\Hom_{\Gamma}(g_j,\Sigma^2
X_i)}&&\Hom_{\cd^b(\mod\Gamma)}(X_j,\Sigma^2
X_i)\ar[u]^{\real}}\] The left vertical map is a bijection and
the right vertical map is injective, so to prove the injectivity of
$\Hom_{\Lambda}(g_j,\Sigma^2 X_i)$ it suffices to prove the
injectivity of $\Hom_{\Gamma}(g_j,\Sigma^2 X_i)$. Writing
\[\Hom_{\cd^b(\mod\Gamma)}(\Sigma X_{ij},\Sigma^2
X_i)=\Ext^1_\Gamma(X_{ij},X_i)=\Hom_\Gamma(\Omega X_{ij},X_i)\] and
\[\Hom_{\cd^b(\mod\Gamma)}(X_j,\Sigma^2 X_i)=\Ext^2_{\Gamma}(X_j,X_i)=\Ext^1_\Gamma(\Omega X_j,X_i)=\Hom_{\Gamma}(\Omega^2 X_j,X_i),\]
we see that
$\Hom_{\Gamma}(g_j,\Sigma^2 X_i)$ is $\Hom_{\Gamma}(\alpha,X_i)$,
where $\alpha$ is defined by the following commutative diagram
\[\xymatrix{0\ar[r]&\Omega^2 X_j\ar[r]\ar[d]^\alpha&P^0\ar[r]\ar[d]^\beta&\Omega X_j\ar[r]\ar[d]^\gamma & 0\\
0\ar[r]&\Omega X_{ij}\ar[r]&Q^0\ar[r] & X_{ij}\ar[r] & 0,}\] Here,
$P^0$ and $Q^0$ are projective covers of $\Omega X_j$ and $X_{ij}$,
respectively, and $\gamma$ is the canonical quotient map. As the
map $\gamma$ is surjective, the map $\beta$ is a split epimorphism.
By the snake lemma, there is {an} exact sequence
\[\xymatrix{\ker(\gamma)\ar[r] & \cok(\alpha)\ar[r]& 0.}\]
Since $X_{ij}$ is the maximal quotient of $\Omega X_j$ in $\cx_i$,
it follows that $\Hom_\Gamma(\ker(\gamma),X_i)=0$, and hence
$\Hom_\Gamma(\cok(\alpha),X_i)=0$. Therefore
$\Hom_\Gamma(\alpha,X_i)$ is injective.
\end{proof}

For two simple-minded collections $\{X_1,\ldots,X_r\}$ and $\{X'_1,\ldots,X'_r\}$ of $\cc$,
define $$\{X_1,\ldots,X_r\}\geq \{X'_1,\ldots,X'_r\}$$ if $\Hom(X'_i,\Sigma^m X_j)=0$ for any $m<0$ and any
$i,j=1,\ldots,r$.

\begin{proposition}\label{p:partial-order-simple-minded}
The relation $\geq$ defined above is a partial order on the set of equivalence classes of simple-minded
collections of $\cc$.
\end{proposition}
\begin{proof}
% We first remark that Proposition~\ref{p:simpleminded-to-t-str} holds for $\cc$.
The reflexivity is clear by the definition
of a simple-minded collection. Next we show the antisymmetry and transitivity.
Let $\{X_1,\ldots,X_r\}$ and $\{X'_1,\ldots,X'_r\}$ be two simple-minded collections of $\cc$ and let
$(\cc^{\leq 0},\cc^{\geq 0})$ and $(\cc'^{\leq 0},\cc'^{\geq 0})$ be the
corresponding $t$-structures
given in Proposition~\ref{p:simpleminded-to-t-str} (the general case). Then
\begin{eqnarray*}
\{X_1,\ldots,X_r\}\geq \{X'_1,\ldots,X'_r\}
&\Leftrightarrow& \Hom(X'_i,\Sigma^m X_j)=0 \text{ for any } m<0 \text{ and }
i,j=1,\ldots,r\\
&\Leftrightarrow&\Hom(\Sigma^{m'}X'_i,\Sigma^m X_j)=0 \text{ for any } m<0,~m'\geq 0 \\
&&\text{ and }
i,j=1,\ldots,r\\
&\Leftrightarrow& \cc'^{\leq 0}\perp \Sigma^{-1}\cc^{\leq 0}\\
&\Leftrightarrow& \cc'^{\leq 0}\subseteq\cc^{\leq 0}.
\end{eqnarray*}

(a) If $\{X_1,\ldots,X_r\}\geq \{X'_1,\ldots,X'_r\}$ and
$\{X'_1,\ldots,X'_r\}\geq \{X_1,\ldots,X_r\}$, then
$(\cc^{\leq 0},\cc^{\geq 0})=(\cc'^{\leq 0},\cc'^{\geq 0})$. In particular, the two
$t$-structures have the same heart. Therefore, both
$\{X_1,\ldots,X_r\}$ and $\{X'_1,\ldots,X'_r\}$ are complete sets of pairwise non-isomorphic
simple objects of the same abelian category,
and hence they are equivalent.

(b) Let $\{X''_1,\ldots,X''_r\}$ be a third simple-minded collection of $\cc$, with corresponding $t$-structure
$(\cc''^{\leq 0},\cc''^{\geq 0})$. Suppose $\{X_1,\ldots,X_r\}\geq \{X'_1,\ldots,X'_r\} $ and
$\{X'_1,\ldots,X'_r\}\geq \{X''_1,\ldots,X''_r\}$. Then $\cc''^{\leq 0}\subseteq\cc'^{\leq 0}\subseteq\cc^{\leq 0}$. Consequently,
$\{X_1,\ldots,X_r\}\geq \{X''_1,\ldots,X''_r\}$.
\end{proof}

\subsection{t-structures}\label{ss:t-str-mutation}

Let $(\cc^{\leq 0},\cc^{\geq 0})$ be a bounded $t$-structure of
$\cc$ such that the heart $\ca$ is a length category which has only
finitely many simple objects $S_1,\ldots,S_r$ up to isomorphism. Then $\{S_1,\ldots, S_r\}$ is a simple-minded collection.
Let
$\cf=\cs_i$ be the extension closure of $S_i$ in $\ca$ and let
$\ct={}^\perp\cs_i$ be the left perpendicular category of $\cs_i$ in
$\ca$. It is easy to show that $(\ct,\cf)$ is a torsion pair of
$\ca$. Define the \emph{left mutation} $\mu^+_i(\cc^{\leq
0},\cc^{\geq 0})=(\cc'^{\leq 0},\cc'^{\geq 0})$ by
\begin{eqnarray*}
 \cc'^{\leq 0}&=&\{M\in\cc\mid H^m(M)=0 \text{ for } m>0 \text{ and } H^0(M)\in\ct\},\\
 \cc'^{\geq 0}&=&\{M\in\cc\mid H^m(M)=0 \text{ for } m<-1 \text{ and } H^{-1}(M)\in \cf\}.
\end{eqnarray*}
Similarly one defines the \emph{right mutation} $\mu^-_i(\cc^{\leq
0},\cc^{\geq 0})$. These mutations provide an effective method to compute the space of Bridgeland's stability conditions on $\cc$ by gluing different charts, see~\cite{Bridgeland05,Woolf10}.

\begin{proposition}\label{p:mutation-t-str}
 The pairs $\mu^+_i(\cc^{\leq
0},\cc^{\geq 0})$ and $\mu^-_i(\cc^{\leq 0},\cc^{\geq 0})$ are
bounded $t$-structures of $\cc$. The heart of $\mu_i^+(\cc^{\leq 0},\cc^{\geq 0})$ has a torsion pair
$(\Sigma\cf,\ct)$ and the heart of $\mu_i^-(\cc^{\leq 0},\cc^{\geq 0})$ has a torsion pair
$(\cs_i^\perp,\Sigma^{-1}\cs_i)$.
Moreover, $\mu^+_i\circ\mu^-_i(\cc^{\leq
0},\cc^{\geq 0})=(\cc^{\leq
0},\cc^{\geq 0})=\mu^-_i\circ\mu^+_i(\cc^{\leq
0},\cc^{\geq 0})$.
\end{proposition}
\begin{proof}
 This follows from~\cite[Proposition 2.1, Corollary 2.2]{HappelReitenSmaloe96} and~\cite[Proposition 2.5]{Bridgeland05}.
\end{proof}

In general the heart of the mutation of a bounded $t$-structure with
length heart is not necessarily a length category. For an example,
let $Q$ be the quiver
\[\xymatrix{1\ar@(ul,dl) & 2\ar[l]}\]
and consider  the bounded derived category $\cc=\cd^b(\nilrep Q)$ of
finite-dimensional nilpotent representations of $Q$. Let $S_1$ and
$S_2$ be the one-dimensional nilpotent representations associated to
the two vertices. Let $\cf=\cs_1$ be the extension closure of $S_1$
and $\ct={}^\perp\cf=\{M\in\nilrep Q\mid \mathrm{top}(M)\in\add(S_2)\}$.
Then the heart $\ca'$ of the left mutation at $1$ of the standard
$t$-structure has a torsion pair $(\Sigma\cf,\ct)$. Due to
$\nilrep Q$ being hereditary, there are no extensions of
$\Sigma\cf$ by $\ct$, and hence any indecomposable object of $\ca'$
belongs to either $\ct$ or $\Sigma \cf$. Suppose that $\ca'$ is a
length category. Then $\ca'$ has two isomorphism classes of simple
modules, which respectively belong to $\ct$ and $\Sigma\cf$, say
$S_2'\in\ct$ and $S_1'\in\Sigma\cf$. For $n\in\mathbb{N}$ define an
indecomposable object $M_n$ in $\ct$ as
\[\xymatrix{k\ar@(ul,dl)_{J_n(0)} && k\ar[ll]_{(0,\ldots,0,1)^{tr}},}\]
where $J_n(0)$ is the (upper triangular) Jordan block of size $n$ and
with eigenvalue $0$. There are no morphisms from $S_1'$ to $M_n$ for any $n$.
Suppose that the Loewy length of $S_2'$ in $\ca$ is $l$.
Then for $n>l$, any morphism from $S_2'$ to $M_n$ factors through
$\mathrm{rad}^{n-l}M_n$ which lies in $\cf$, and hence the morphism has to be
zero. Therefore $M_n$ ($n>l$), considered as an object in $\ca'$, does not have finite length,
a contradiction.

\medskip

For two bounded $t$-structures $(\cc^{\leq 0},\cc^{\geq 0})$ and $(\cc'^{\leq 0},\cc'^{\geq 0})$ on $\cc$,
define $$(\cc^{\leq 0},\cc^{\geq 0})\geq (\cc'^{\leq 0},\cc'^{\geq 0})$$ if $\cc^{\leq 0}\supseteq \cc'^{\leq 0}$.
This defines a partial order on the set of bounded $t$-structures on $\cc$.

\subsection{Co-t-structures}\label{ss:co-t-str-mutation}

Let $(\cc_{\geq 0},\cc_{\leq 0})$ be a bounded co-$t$-structure of
$\cc$. Assume that the co-heart admits a basic additive generator
$M=M_1\oplus\ldots\oplus M_r$ with $M_i$ indecomposable. Then $M$ is
a silting object of $\cc$. Let $i=1,\ldots,r$. Define $\cc'_{\leq
0}$ as the additive closure of the extension closure of $\Sigma ^m M_j$,
$j\neq i$, and $\Sigma^{m+1} M_i$ for $m\geq 0$ and define
$\cc'_{\geq 0}$ as the left perpendicular category of $\Sigma\cc'_{\leq
0}$. The \emph{left mutation} $\mu^+_i(\cc_{\geq 0},\cc_{\leq 0})$
is defined as the pair $(\cc'_{\geq 0},\cc'_{\leq 0})$. Similarly
one defines the \emph{right mutation} $\mu^-_i(\cc_{\geq
0},\cc_{\leq 0})$.

\begin{proposition}\label{p:mutation-co-t-str}
The pairs $\mu^+_i(\cc_{\geq 0},\cc_{\leq 0})$ and
$\mu^-_i(\cc_{\geq 0},\cc_{\leq 0})$ are bounded co-$t$-structures
on $\cc$. Moreover, $\mu^+_i\circ\mu^-_i(\cc_{\geq 0},\cc_{\leq
0})=(\cc_{\geq 0},\cc_{\leq 0})=\mu^-_i\circ\mu^+_i(\cc_{\geq
0},\cc_{\leq 0})$.
\end{proposition}
\begin{proof} This can be proved directly. Here we alternatively make use of
the results in Sections~\ref{ss:silting-obj} and~\ref{ss:silting-obj-mutation}.
Recall from Theorem~\ref{t:mutation-silting} that there is a mutated
silting object $\mu_i^+(M)$. It is straightforward to check, using the defining
triangle for $\mu_i^+(M)$, that $\mu_i^+(\cc_{\geq 0},\cc_{\leq 0})$
is the bounded co-$t$-structure associated to $\mu_i^+(M)$ as
defined in Proposition~\ref{p:AI's-from-silting-to-co-t-str}, and
similarly for $\mu_i^-$. The second statement follows
from Theorem~\ref{t:mutation-silting}.
\end{proof}

For two bounded co-$t$-structures $(\cc_{\geq 0},\cc_{\leq 0})$ and $(\cc'_{\geq 0},\cc'_{\leq 0})$ on $\cc$, define
$$(\cc_{\geq 0},\cc_{\leq 0})\geq(\cc'_{\geq 0},\cc'_{\leq 0})$$ if $\cc_{\leq 0}\supseteq\cc'_{\leq 0}$.
This defines a partial order on the set of bounded co-$t$-structures on $\cc$.

\subsection{The bijections commute with mutations}\label{ss:commuting-with-mutations}

Let $\Lambda$ a
finite-dimensional algebra over $K$.

\begin{theorem}\label{t:main-theorem-commuting-with-mutations}
\label{maintheoremmutations}
The $\phi_{ij}$'s defined in Section~\ref{s:the-maps} commute with
the left and right mutations defined in previous subsections.
\end{theorem}

A priori it it not known that the heart of the mutation
of a bounded $t$-structure with length heart is again a length
category. So the theorem becomes well-stated only when the proof
has been finished.

\begin{proof}

In view of Lemma~\ref{l:transitivity},
Theorem~\ref{t:mutation-silting}, and
Propositions~\ref{p:mutation-simple-minded}, ~\ref{p:mutation-t-str}
and~\ref{p:mutation-co-t-str}, it suffices to prove that
$\phi_{41}$, $\phi_{31}$ and $\phi_{23}$ commute with the
corresponding left mutations.

(a) $\phi_{41}$ commutes with $\mu^+_i$: this was already shown in
the proof of Proposition~\ref{p:mutation-co-t-str}.

(b) $\phi_{31}$ commutes with $\mu^+_i$: Let
$M=M_1\oplus\ldots\oplus M_r$ be a silting object with $M_i$
indecomposable and $(\cc^{\leq 0},\cc^{\geq 0})=\phi_{31}(M)$. We want to show $\mu_i^+(\cc^{\leq 0},\cc^{\geq 0})=\phi_{31}(\mu_i^+(M))$. 

Let $\tilde{\Gamma}$ be the truncated dg endomorphism algebra of $M$ as in Section~\ref{ss:silting-der-equiv}. Then there is a triangle equivalence $F=?\lten_{\tilde{\Gamma}}M:\cd_{fd}(\tilde{\Gamma})\rightarrow\cd^b(\mod \Lambda)$, which takes $\tilde{\Gamma}$ to $M$ and takes the standard $t$-structure $(\cd^{\leq 0},\cd^{\geq 0})$ on $\cd_{fd}(\tilde{\Gamma})$ to $(\cc^{\leq 0},\cc^{\geq 0})$. There is a decomposition $1=e_1+\ldots+e_r$, where $e_1,\ldots,e_r$ are (not necessarily primitive) idempotents of $\tilde{\Gamma}$ such that $F$ takes $e_j\tilde{\Gamma}$ to $M_j$ for $1\leq j\leq r$. 

Let $\Gamma=H^0(\tilde{\Gamma})$ and $\pi:\tilde{\Gamma}\rightarrow\Gamma$ be the canonical projection. By abuse of notation, write $e_1=\pi(e_1),\ldots,e_r=\pi(e_r)$. Then $e_1\Gamma,\ldots,e_r\Gamma$ are indecomposable projective $\Gamma$-modules. Let $S_1,\ldots,S_r$ be the corresponding simple modules. Recall that the heart of the $t$-structure $(\cd^{\leq 0},\cd^{\geq 0})$ is $\mod\Gamma$.
Let $\cf=\add(S_i)\subseteq \mod\Gamma$ and $\ct={}^\perp S_i$. Define $\cd'^{\leq 0}$ (respectively, $\cd'^{\geq 0}$) to be the extension closure of $\Sigma\cd^{\leq 0}$ and $\ct$ (respectively, of $\Sigma\cf$ and $\cd^{\geq 0}$). Then $F(\cd'^{\leq 0},\cd'^{\geq 0})=\mu_i^+(\cc^{\leq 0},\cc^{\geq 0})$.

The left mutation of $\tilde{\Gamma}$ at $e_i\tilde{\Gamma}$ is  $\mu_i^+(\tilde{\Gamma})=Q_i\oplus\bigoplus_{j\neq i}e_j\tilde{\Gamma}$, where $Q_i$ is defined by the triangle
\begin{eqnarray}\label{eq:triangle-commutativity}
\xymatrix{e_i\tilde{\Gamma}\ar[r]^f & E\ar[r] & P_i'\ar[r] & \Sigma e_i\tilde{\Gamma}},
\end{eqnarray}
where $f$ is a minimal left $\add(\bigoplus_{j\neq i}e_j\tilde{\Gamma})$-approximation. Then $F(\mu_i^+(\tilde{\Gamma}))=\mu_i^+(M)$. 
Define 
\begin{eqnarray*}
\cd''^{\leq 0} &=& \{N\in \cd_{fd}(\tilde{\Gamma})~|~\Hom(\mu_i^+(\tilde{\Gamma}),\Sigma^m N)=0, \forall m>0\},\\
\cd''^{\geq 0} &=& \{N\in \cd_{fd}(\tilde{\Gamma})~|~\Hom(\mu_i^+(\tilde{\Gamma}),\Sigma^m N)=0, \forall m<0\}.
\end{eqnarray*}

Thus  showing $\mu_i^+(\cc^{\leq 0},\cc^{\geq 0})=\phi_{31}(\mu_i^+(M))$ is equivalent to showing the equality $(\cd'^{\leq 0},\cd'^{\geq 0})=(\cd''^{\leq 0},\cd''^{\geq 0})$, equivalently, the inclusions $\cd'^{\leq 0}\subseteq \cd''^{\leq 0}$ and $\cd'^{\geq 0}\subseteq \cd''^{\geq 0}$. It suffices to prove $\ct\subseteq \cd''^{\leq 0}$, $\Sigma\cd^{\leq 0}\subseteq \cd''^{\leq 0}$, $\Sigma\cf\subseteq\cd''^{\geq 0}$ and $\cd^{\geq 0}\subseteq\cd''^{\geq 0}$. We only show the first inclusion, the other three are easy.

Let $T\in\ct$. To show $T\in \cd''^{\leq 0}$, it suffices to show $\Hom(Q_i,\Sigma T)=0$. Applying $\Hom(?,T)$ to the triangle (\ref{eq:triangle-commutativity}), we obtain a long exact sequence
\[
\xymatrix{
\Hom(E, T)\ar[r]^{f^*} & \Hom(e_i\tilde{\Gamma},\Sigma T)\ar[r] &\Hom(Q_i,\Sigma T)\ar[r] &\Hom(E,\Sigma T)=0
}.
\]
We claim that $f^*$ is surjective. Then the desired result follows. Consider the commutative diagram
\begin{eqnarray}\label{d:commutativity}
\begin{split}
\xymatrix@C=3.5pc{
\Hom(e_i\Gamma,T)\ar[r]^{\pi_i^*} & \Hom(e_i\tilde{\Gamma},T)\\
\Hom(H^0(E),T)\ar[r]^{\pi_E^*}\ar[u]_{H^0(f)^*} & \Hom(E,T),\ar[u]_{f^*}
}
\end{split}
\end{eqnarray}
where $\pi_i:e_i\tilde{\Gamma}\rightarrow e_i\Gamma$ and $\pi_E:E\rightarrow H^0(E)$ are the canonical projections. Let $C=\ker(\pi_i)$. Then  there is a triangle
\[
\xymatrix{
C\ar[r] & e_i\tilde{\Gamma}\ar[r]^{\pi_i} & e_i\Gamma \ar[r] &\Sigma C
}.
\]
Note that $C$ belongs to $\Sigma\cd^{\leq 0}$, which implies that $\Hom(C,T)=0=\Hom(\Sigma C,T)$. 
It follows that the map $\pi_i^*$ is bijective. Similarly, the map $\pi_E^*$ is also bijective. Thus it suffices to show the surjectivity of $H^0(f)^*$. Now let $P_T$ be a projective cover of $T$ in $\mod \Gamma$. Then $P_T$ belongs to $\add(\bigoplus_{j\neq i}e_j\Gamma)$ because $T\in\ct={}^\perp S_i$. It follows that any morphism $e_i\Gamma\rightarrow T$ factors through $P_T$, and hence factors through $H^0(f):e_i\Gamma\rightarrow H^0(E)$, since $H^0(f)$ is a minimal left  $\add(\bigoplus_{j\neq i}e_j\Gamma)$-approximation (for $H^0|_{\add(\tilde{\Gamma})}:\add(\tilde{\Gamma})\rightarrow \add(\Gamma)$ is an equivalence). This shows that $H^0(f)^*$ is surjective, completing the proof of the claim.

%Recall that $(\cc^{\leq 0},\cc^{\geq 0})$ is a bounded $t$-structure on $\cd^b(\mod\Lambda)$ with length heart, whose simple objects $X_1,\ldots,X_r$ can be ordered such that for $1\leq i,j\leq r$ and $m\in\mathbb{Z}$
%\[\Hom(M_i,\Sigma^m X_j)=\begin{cases} {{}_{R_j}%R_j} & \text{ if } i=j \text{ and } m=0,\\ 0 & \text{ otherwise}.\end{cases}\]
%{Here $R_j$ is the endomorphism algebra of $X_j$ for $1\leq j\leq r$.}

%Let $i=1,\ldots,r$ be fixed. Let $\cf$ be the extension closure of $X_i$ in $\cd^b(\mod\Lambda)$ and let $\ct={}^\perp\cf$ be its left
%orthogonal category in the heart. Let  $(\cc'^{\leq
%0},\cc'^{\geq 0})=\mu^+_i(\cc^{\leq 0},\cc^{\geq 0})$. By definition, $\cc'^{\leq 0}$ respectively $\cc'^{\geq 0}$
%is the extension closure of $\Sigma^m \ct$ and $\Sigma^{m+1}\cf$ for $m\geq 0$ respectively for $m\leq 0$.

%Write
%$\mu^+_i(M)=M'=M_i'\oplus\bigoplus_{j\neq i}M_j$. By definition there is a triangle
%\[\xymatrix{M_i\ar[r]^f&E\ar[r]&M_i'\ar[r]&\Sigma M_i}\]
%where $f$ is a left $\add(\bigoplus_{j\neq i}M_j)$-approximation. One easily sees that
%\begin{eqnarray*}
% \cc'^{\leq 0}&\subseteq&\{N\in\cd^b(\mod\Lambda)\mid \Hom(M',\Sigma^m N)=0,~\forall~m>0\}\\
% \cc'^{\geq 0}&\subseteq&\{N\in\cd^b(\mod\Lambda)\mid \Hom(M',\Sigma^m N)=0,~\forall~m<0\}
%\end{eqnarray*}
%The two subcategories on the right hand side are respectively the aisle and co-aisle of $\phi_{31}(M')$. It follows that the above inclusions are both equalities. This completes the proof for
%$\phi_{31}\circ\mu^+_i=\mu^+_i\circ\phi_{31}$.

(c) $\phi_{23}$ commutes with $\mu^+_i$: Let $(\cc^{\leq
0},\cc^{\geq 0})$ be a bounded $t$-structure on $\cd^b(\mod\Lambda)$
with length heart. Let $\{X_1,\ldots,X_r\}=\phi_{23}(\cc^{\leq
0},\cc^{\geq 0})$. The mutated simple-minded
collection $\mu^+_i(X_1,\ldots,X_r)$ is contained in the heart of the mutated
$t$-structure $\mu^+_i(\cc^{\leq
0},\cc^{\geq 0})$. Consequently, the aisle and co-aisle of $\phi_{32}\circ\mu^+_i(X_1,\ldots,X_r)$
are respectively contained in the aisle and co-aisle of $\mu^+_i(\cc^{\leq
0},\cc^{\geq 0})$, and hence $\phi_{32}\circ\mu^+_i(X_1,\ldots,X_r)=\mu^+_i(\cc^{\leq
0},\cc^{\geq 0})$, \ie $\phi_{23}\circ\mu^+_i(\cc^{\leq
0},\cc^{\geq 0})=\mu^+_i\circ\phi_{23}(\cc^{\leq
0},\cc^{\geq 0})$.
\end{proof}

%\begin{corollary}
% The left and right mutations of a bounded $t$-structure on $\cd^b(\mod\Lambda)$ with length heart have length hearts.
%\end{corollary}

\subsection{The bijections are isomorphisms of partially ordered sets}

Let $\Lambda$ be a finite-dimensional algebra over $K$.
\begin{theorem}\label{t:main-thm-partially-ordered-sets}
 The $\phi_{ij}$'s defined in Section~\ref{section-maps} are isomorphisms of partially ordered sets with respect to
 the partial orders defined in previous subsections.
\end{theorem}
\begin{proof} In view of Theorem~\ref{t:main-thm-bijections} and Lemma~\ref{l:transitivity},
it suffices to show that $f(x)\geq f(y)$ if and only if $x\geq y$ for
$f=\phi_{41}, \phi_{32}$ and $\phi_{34}$.

(a) For $\phi_{41}$ the desired result follows from~\cite[Proposition {2.14}]{AiharaIyama12}.

(b) For $\phi_{32}$ the desired result is included in the proof of Proposition~\ref{p:partial-order-simple-minded}.

(c) Let $(\cc_{\geq 0},\cc_{\leq 0})$ and $(\cc'_{\geq 0},\cc'_{\leq 0})$
be two bounded co-$t$-structures
on $\cc$ and let $(\cc^{\leq 0},\cc^{\geq 0})$ and $(\cc'^{\leq 0},\cc'^{\geq 0})$ be
their respective images under $\phi_{34}$.
Then by definition
\begin{eqnarray*}
 \cc^{\leq 0}&=&\{M\in\cd^b(\mod \Lambda)\mid\Hom(N,M)=0 \; \forall N\in\Sigma^{-1}\cc_{\geq 0}\},\\
 \cc'^{\leq 0}&=&\{M\in\cd^b(\mod \Lambda)\mid\Hom(N,M)=0 \; \forall N\in\Sigma^{-1}\cc'_{\geq 0}\}.
\end{eqnarray*}
Here, $\cc^{\leq 0}\supseteq\cc'^{\leq 0}$ if and only if $\cc_{\geq 0}\supseteq \cc'_{\geq 0}$, and hence by definition
$(\cc^{\leq 0},\cc^{\geq 0})\geq (\cc'^{\leq 0},\cc'^{\geq 0})$ if and only if
$(\cc_{\geq 0},\cc_{\leq 0})\geq(\cc'_{\geq 0},\cc'_{\leq 0})$.
\end{proof}

\section{A concrete example} \label{section-example}

Let $\Lambda$ be the finite-dimensional $K$-algebra given by the quiver
$$\xymatrix{1\ar@<.7ex>[r]^{\alpha}&2\ar@<.7ex>[l]^{\beta}}$$
with relation $\alpha\beta=0$. This algebra has many manifestations:
It is, possibly up to Morita equivalence, the Auslander algebra of
$k[x]/x^2$, the Schur algebra $S(2,2)$ ($\mathrm{char}K=2$) and the
principal block of the category $\co$ for
$\mathfrak{sl}_2(\mathbb{C})$ ($K=\mathbb{C}$). In this section we
will compute the derived Picard group for $\Lambda$ and classify all
silting objects/simple-minded collections in $\cd^b(\mod\Lambda)$.
As a consequence of this classification and a result of
Woolf~\cite[Theorem 3.1]{Woolf10}, the space of stability conditions
on $\cd^b(\mod\Lambda)$ is exactly $\mathbb{C}^2$.

\subsection{Indecomposable objects}
Let $P_1$ and $P_2$ be the indecomposable projective $\Lambda$-modules corresponding to the vertices $1$ and $2$.
Then up to isomorphism and up to shift an indecomposable object in $\cd^b(\mod\Lambda)$ belongs to one of the following four families
(see for example~\cite{BurbanDrozd06,BekkertMerklen03})
\begin{itemize}
\item[$\cdot$] $P_1(n)=P_1\rightarrow P_1\rightarrow\ldots\rightarrow P_1\rightarrow
P_1$, $n\geq 1$,
\item[$\cdot$] $R(n)=P_1\rightarrow P_1\rightarrow\ldots\rightarrow P_1\rightarrow P_1\rightarrow
P_2$, $n\geq 0$,
\item[$\cdot$] $L(n)=P_2\rightarrow P_1\rightarrow P_1\rightarrow\ldots\rightarrow P_1\rightarrow
P_1$, $n\geq 0$,
\item[$\cdot$] $B(n)=P_2\rightarrow P_1\rightarrow P_1\rightarrow\ldots\rightarrow P_1\rightarrow P_1\rightarrow
P_2$, $n\geq 1$,
\end{itemize}
where the homomorphisms are the unique non-isomorphisms, $n$ is the
number of occurrences of $P_1$ and the rightmost components have been put
in degree $0$.

\subsection{The Auslander--Reiten quiver}

The Auslander--Reiten quiver of $\cd^b(\mod\Lambda)$ consists of
three components: two $\mathbb{Z}A_\infty$ components and one
$\mathbb{Z}A^\infty_\infty$ component
(see~\cite{BobinskiGeissSkowronski04,KalckYang11})
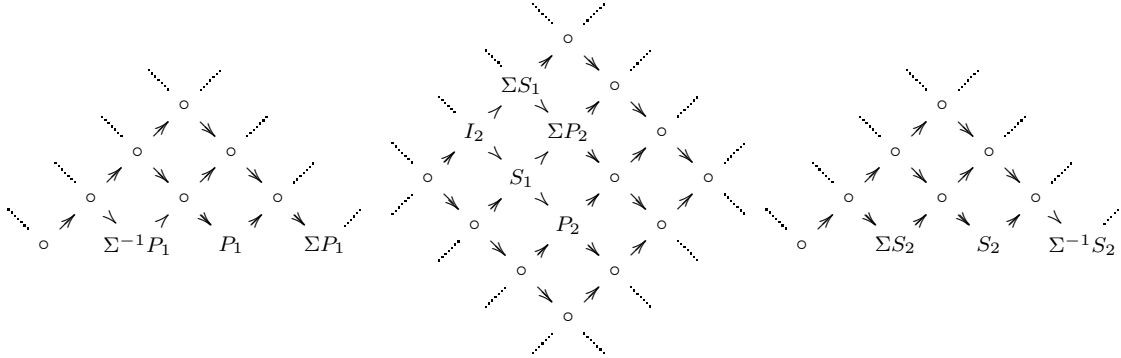
\vspace{-20pt}\[
\begin{picture}(400,160)
\put(-20,120){{\scriptsize
\begin{xy} 0;<0.35pt,0pt>:<0pt,-0.35pt>::
(0,150) *+{} ="0", (0,50) *+{} ="1", (50,200) *+{\circ} ="2",
(50,100)
*+{} ="3", (50,0) *+{} ="4", (100,150) *+{\circ} ="5", (100,50) *+{}
="6", (150,200) *+{\Sigma^{-1}P_1} ="7", (150,100) *+{\circ} ="8",
(150,0)
*+{} ="9", (200,150) *+{\circ} ="10", (200,50)
*+{\circ} ="11", (250,200) *+{P_1} ="12", (250,100) *+{\circ} ="13",
(250,0) *+{} ="14", (300,150) *+{\circ} ="15", (300,50) *+{} ="16",
(350,200) *+{\Sigma P_1} ="17", (350,100) *+{} ="18", (400,150)
*+{} ="20",  "0", {\ar@{.}"2"}, "1",
{\ar@{}"3"}, "1", {\ar@{}"4"}, "2", {\ar"5"}, "3", {\ar@{.}"5"},
"4", {\ar@{}"6"}, "5", {\ar"7"}, "5", {\ar"8"}, "6", {\ar@{.}"8"},
"7", {\ar"10"}, "8", {\ar"10"}, "8", {\ar"11"}, "9", {\ar@{.}"11"},
"10", {\ar"12"}, "10", {\ar"13"}, "11", {\ar"13"}, "11",
{\ar@{.}"14"}, "12", {\ar"15"}, "13", {\ar"15"}, "13",
{\ar@{.}"16"}, "15", {\ar"17"}, "15", {\ar@{.}"18"}, "17",
{\ar@{.}"20"},
\end{xy}
} } \put(120,145){ {\scriptsize
\begin{xy} 0;<0.35pt,0pt>:<0pt,-0.35pt>::
(0,250) *+{}="100", (0,150) *+{}="101", (50,300) *+{}="102",
(50,200)*+{\circ}="103", (50,100) *+{}="104", (100,350)
*+{}="105",(100,250) *+{\circ}="106",(100,150) *+{I_2}="107",
(100,50) *+{}="108", (150,400)*+{}="109", (150,300)*+{\circ}="110",
(150,200)*+{S_1}="111", (150,100)*+{\Sigma S_1}="112",
(150,0)*+{}="113", (200,350)*+{\circ}="114", (200,250)*+{P_2}="115",
(200,150)*+{\Sigma P_2}="116",
(200,50)*+{\circ}="117",(250,400)*+{}="118",
(250,300)*+{\circ}="119", (250,200)*+{\circ}="120",
(250,100)*+{\circ}="121", (250,0)*+{}="122", (300,350)
*+{}="123",(300,250) *+{\circ}="124",(300,150) *+{\circ}="125",
(300,50) *+{}="126", (350,300) *+{}="127", (350,200)*+{\circ}="128",
(350,100)
*+{}="129", (400,250) *+{}="130", (400,150) *+{}="131",
"100", {\ar@{.}"103"}, "101", {\ar@{.}"103"}, "102", {\ar@{.}"106"},
"103", {\ar"106"}, "103", {\ar"107"}, "104", {\ar@{.}"107"}, "105",
{\ar@{.}"110"},"106",{\ar"110"},"106", {\ar"111"},"107",
{\ar"111"},"107", {\ar"112"},"108", {\ar@{.}"112"},"109",
{\ar@{.}"114"},"110", {\ar"114"}, "110", {\ar"115"},"111",
{\ar"115"},"111", {\ar"116"},"112", {\ar"116"},"112",
{\ar"117"},"113", {\ar@{.}"117"},"114", {\ar@{.}"118"},"114",
{\ar"119"},"115", {\ar"119"}, "115", {\ar"120"}, "116",
{\ar"120"},"116", {\ar"121"},"117", {\ar"121"},"117",
{\ar@{.}"122"},"119", {\ar@{.}"123"},"119", {\ar"124"}, "120",
{\ar"124"},"120", {\ar"125"},"121", {\ar"125"},"121",
{\ar@{.}"126"}, "124", {\ar@{.}"127"},"124", {\ar"128"},"125",
{\ar"128"},"125", {\ar@{.}"129"},"128", {\ar@{.}"130"}, "128",
{\ar@{.}"131"},
\end{xy} }
} \put(260,120){ {\scriptsize
\begin{xy} 0;<0.35pt,0pt>:<0pt,-0.35pt>::
(0,150) *+{} ="0", (0,50) *+{} ="1", (50,200) *+{\circ} ="2",
(50,100)
*+{} ="3", (50,0) *+{} ="4", (100,150) *+{\circ} ="5", (100,50) *+{}
="6", (150,200) *+{\Sigma S_2} ="7", (150,100) *+{\circ} ="8",
(150,0)
*+{} ="9", (200,150) *+{\circ} ="10", (200,50)
*+{\circ} ="11", (250,200) *+{S_2} ="12", (250,100) *+{\circ} ="13",
(250,0) *+{} ="14", (300,150) *+{\circ} ="15", (300,50) *+{} ="16",
(350,200) *+{\Sigma^{-1}S_2} ="17", (350,100) *+{} ="18", (400,150)
*+{} ="20",  "0", {\ar@{.}"2"}, "1",
{\ar@{}"3"}, "1", {\ar@{}"4"}, "2", {\ar"5"}, "3", {\ar@{.}"5"},
"4", {\ar@{}"6"}, "5", {\ar"7"}, "5", {\ar"8"}, "6", {\ar@{.}"8"},
"7", {\ar"10"}, "8", {\ar"10"}, "8", {\ar"11"}, "9", {\ar@{.}"11"},
"10", {\ar"12"}, "10", {\ar"13"}, "11", {\ar"13"}, "11",
{\ar@{.}"14"}, "12", {\ar"15"}, "13", {\ar"15"}, "13",
{\ar@{.}"16"}, "15", {\ar"17"}, "15", {\ar@{.}"18"}, "17",
{\ar@{.}"20"},
\end{xy}
} } \end{picture}\]

The abelian category $\mod\Lambda$ has five indecomposable objects
up to isomorphism: the two simple modules $S_1$ and $S_2$, their
projective covers $P_1$ and $P_2$ and their injective envelopes
$I_1=P_1$ and $I_2$. They are marked on the above Auslander--Reiten
quiver.

The left $\mathbb{Z}A_\infty$ component consists of shifts of
$P_1(n)$, $n\geq 1$. The Auslander--Reiten translation $\tau$ takes
$P_1(n)$ to $\Sigma^{-1}P_1(n)$. It is straightforward to check that
$P_1$ is a $0$-spherical object of $\cd^b(\mod\Lambda)$ in the sense
of Seidel and Thomas~\cite{SeidelThomas01}. The additive closure of
this component is the triangulated subcategory generated by $P_1$.
This component will be referred to as the $0$-spherical component.

The $\mathbb{Z}A^\infty_\infty$ component consists of shifts of
$R(n)$ and $L(n)$, $n\geq 0$. Note that $S_1=L(1)$, $P_2=R(0)=L(0)$
and $I_2=L(2)$. The Auslander--Reiten translation $\tau$ takes
$R(n)$ ($n\geq 2$) to $\Sigma R(n-2)$, takes $R(1)$ to $L(1)$ and
takes $L(n)$ to $\Sigma^{-1}L(n+2)$.

The right $\mathbb{Z}A_\infty$ component consists of shifts of
$B(n)$, $n\geq 1$. The Auslander--Reiten translation takes $B(n)$ to
$\Sigma B(n)$. The simple module $S_2=B(1)$ is a $2$-spherical
object of $\cd^b(\mod\Lambda)$ and the additive closure of this
component is the triangulated subcategory generated by $S_2$. This
component will be referred to as the $2$-spherical component.

\subsection{The derived Picard group}

Let $E$ be a spherical object of a triangulated category $\cc$ in
the sense of Seidel and Thomas~\cite{SeidelThomas01}. Then the twist
functor $\Phi_{E}$ defined by
\[\Phi_{E}(M)=\mathrm{Cone}(\bigoplus_{m\in\mathbb{Z}}\Hom(\Sigma^m E,M)\ten \Sigma^m E\stackrel{ev}{\longrightarrow} M),\]
where $ev$ is the evaluation map, is an auto-equivalence of $\cc$
by~\cite[Proposition 2.10]{SeidelThomas01}.

Recall from the preceding subsection that $P_1$ is a $0$-spherical
object and $S_2$ is a $2$-spherical object of $\cd^b(\mod\Lambda)$.
Thus the associated twist functors $\Phi_{P_1}$ and $\Phi_{S_2}$ are
two  auto-equivalences of $\cd^b(\mod\Lambda)$.

\begin{lemma} For $M$ in $\cd^b(\mod\Lambda)$ there are isomorphisms
$\Phi_{S_2}(M)\cong\Phi_{P_1}\circ\Sigma^{-1}(M)$ and
$\Phi_{P_1}^2(M)\cong\nu^{-1}\circ\Sigma^2(M)$. Moreover, if $M$ is indecomposable and belongs to the
$\mathbb{Z}A^\infty_\infty$ component, there exists a unique pair of
integers $(n,n')$ such that $M\cong
\Phi_{P_1}^n\circ\Phi_{S_2}^{n'}(P_2)$.
\end{lemma}
\begin{proof} Observe that $\Phi_{P_1}(S_1)\cong\Sigma P_2$, $\Phi_{P_1}(P_1)\cong \Sigma P_1$, $\Phi_{P_1}(S_2)\cong S_2$
and $\Phi_{S_2}(S_1)\cong P_2$, $\Phi_{S_2}(P_1)\cong P_1$, $\Phi_{S_2}(S_2)\cong \Sigma^{-1}S_2$.
Since auto-equivalences preserve the shape of the Auslander--Reiten quiver, the statements follow.
\end{proof}

\begin{remark}
 Inspecting the action of $\Phi_{P_1}$ and $\Phi_{S_2}$ on maps shows that the isomorphism
$\Phi_{P_1}^2(M)\cong\nu^{-1}\circ\Sigma^2(M)$ is functorial, while $\Phi_{S_2}(M)\cong\Phi_{P_1}\circ\Sigma^{-1}(M)$ is not.
\end{remark}

Let $\Aut\cd^b(\mod\Lambda)$ denote the group of algebraic
auto-equivalences of $\cd^b(\mod\Lambda)$, \ie those which admits a
dg lift. By~\cite[Lemma 6.4]{Keller94}, such an auto-equivalence is
naturally isomorphic to the derived tensor functor of a complex of
bimodules.

\begin{lemma}
$\Aut\cd^b(\mod\Lambda)$ is isomorphic to $\mathbb{Z}^2\times
K^{\times}$.
\end{lemma}
\begin{proof} Let $F\in\Aut\cd^b(\mod\Lambda)$. Since $F$ preserves
the Auslander--Reiten quiver, the object $F(P_2)$ is in the
$\mathbb{Z}A_\infty^\infty$ component. Thus there is a pair of
integers $(n_F,n'_F)$ such that $F(P_2)\cong
\Phi_{P_2}^{n_F}\circ\Phi_{S_2}^{n'_F}(P_2)$. This allows us to
define a map
\[\xymatrix@R=0.1pc{f:&\Aut\cd^b(\mod\Lambda)\ar[r]&\mathbb{Z}^2\\
& F\ar@{|->}[r] &(n_F,n'_F).}\] This map is clearly a
surjective group homomorphism. Moreover, the group homomorphism
\[\xymatrix@R=0.1pc{\mathbb{Z}^2\ar[r]&\Aut\cd^b(\mod\Lambda)\\
(n,n')\ar@{|->}[r] &\Phi_{P_2}^{n}\circ\Phi_{S_2}^{n'}}\] is a
retraction of $f$. Therefore
$\Aut\cd^b(\mod\Lambda)\cong\mathbb{Z}^2\times \ker(f)$.

Let $F\in\ker(f)$. Then $F(P_2)\cong P_2$. This forces $F(P_1)\cong
P_1$, and hence $F$ is induced from an outer automorphism of
$\Lambda$ which fixes the two primitive idempotents $e_1$ and $e_2$.
Thus $\ker(f)\cong K^\times$, finishing the proof.
\end{proof}

\subsection{Morphism spaces}

We first compute the morphism spaces between the two $\mathbb{Z}A_\infty$ components.
\begin{lemma}\label{l:morphism-on-spherical-comp}
\begin{itemize}
\item[(a)] For $n\geq 2$, $\Hom(P_1(n),\Sigma^m
P_1(n))$ does not vanish for some $m>0$ and for some $m<0$. For
$n=1$, $\Hom(P_1,\Sigma^m P_1)$ is isomorphic to $K[x]/x^2$ for
$m=0$ and vanishes for $m\neq 0$.
\item[(b)] For $n\geq 2$, $\Hom(B(n),\Sigma^m
B(n))$ does not vanish for some $m>0$ and for some $m<0$. For $n=1$,
$\Hom(S_2,\Sigma^m S_2)$ is $K$ for $m=0, 2$ and vanishes for $m\neq
0,2$.
%\item[(c)] For $n,n'\geq 1$ and $m\in\mathbb{Z}$, $\Hom_{\cd^b(\Lambda)}(M(1,n),\Sigma^m M(4,n'))$ and
%$\Hom_{\cd^b(\Lambda)}(M(4,n'),\Sigma^m M(1,n))$ vanish.
\end{itemize}
\end{lemma}
\begin{proof}
Direct computation, or apply some general result (e.g.~\cite[Section
2]{HolmJorgensenYang11}) to the triangulated categories generated by
$P_1$ and $S_2$.
%(c) This proof is learned from Martin Kalck. We have $\nu M(1,n)=M(1,n)$ and $\nu M(4,n')=\Sigma^2 M(4,n')$. Therefore applying the Auslander--Reiten
%formula twice we obtain
%\begin{eqnarray*}
% \Hom_{\cd^b(\mod\Lambda)}(M(1,n),\Sigma^m M(4,n')) &\cong&
%\mathrm{D}\Hom_{\cd^b(\mod\Lambda)}(\Sigma^m M(4,n'),\nu M(1,n))\\
%&\cong&\mathrm{D}\Hom_{\cd^b(\mod\Lambda)}(\Sigma^m M(4,n'),M(1,n))\\
%&\cong&\Hom_{\cd^b(\mod\Lambda)}(M(1,n),\Sigma^m\nu M(4,n'))\\
%&\cong&\Hom_{\cd^b(\mod\Lambda)}(M(1,n),\Sigma^{m+2} M(4,n')).
%\end{eqnarray*}
%It follows by induction that this has to vanish.
\end{proof}

Next we compute the morphism spaces between $P_2$ and the objects on the $\mathbb{Z}A^\infty_\infty$ component.

\begin{lemma}\label{l:morphism-on-non-spherical-comp} Let $n\geq 0$.
\begin{itemize}
\item[(a)]
$\Hom(P_2,\Sigma^m R(n))$ is $K$ if $-n\leq m\leq 0$ and is $0$
otherwise.
\item[(a')]
$\Hom(R(n),\Sigma^m P_2)$ is $K$ if $2\leq m\leq n$ or if $n=0, m=0$
and is $0$ otherwise.
\item[(b)]
$\Hom(P_2,\Sigma^m L(n))$ is $K$ if $2-n\leq m\leq 0$ or if $n=0,
m=0$ and is $0$ otherwise.
\item[(b')]$\Hom(L(n),\Sigma^m P_2)$ is $K$
if $0\leq m\leq n$ and is $0$ otherwise.
\end{itemize}
\end{lemma}
\begin{proof}
(a) and (b) Because $\Hom(P_2,M)=H^0(M)e_2$.

(a') and (b') are obtained from (a) and (b) by applying the
Auslander--Reiten formula
$\mathrm{D}\Hom(M,N)\cong\Hom(N,\tau\Sigma
M)$.
\end{proof}

\subsection{Silting objects and simple-minded collections}

Now we are ready to classify the silting objects and simple-minded collections in $\cd^b(\mod\Lambda)$.

\begin{proposition}\label{p:example-silting}
Up to isomorphism, any basic silting object of $\cd^b(\mod\Lambda)$
belongs to one of the following two families
\begin{itemize}
\item[$\cdot$] $\Phi_{P_1}^n\circ\Phi_{S_2}^{n'}(P_1\oplus P_2)$, $n,n'\in\mathbb{Z}$,
the corresponding simple-minded collection is
$\Phi_{P_1}^n\circ\Phi_{S_2}^{n'}\{S_1,S_2\}$,
\item[$\cdot$] $\Phi_{P_1}^n\circ\Phi_{S_2}^{n'}(\Sigma^m S_1\oplus
P_2)$, $n,n'\in\mathbb{Z}$ and $m\leq -1$, the corresponding simple-minded collection is
$\Phi_{P_1}^n\circ\Phi_{S_2}^{n'}\{\Sigma^m S_1,I_2\}$.
\end{itemize}
\end{proposition}
\begin{proof} Let $N$ be an indecomposable direct summand of a silting object.
By Lemma~\ref{l:morphism-on-spherical-comp}, $N$ does not belong to the $2$-spherical component, and $N$ belongs to the
$0$-spherical component if and only if $N$ is a shift of $P_1$.
Moreover, a basic silting object can have at most one shift of $P_1$ as a direct summand.
It follows that a silting object has at least one indecomposable direct summand from the $\mathbb{Z}A_\infty^\infty$ component.

Let $M=M_1\oplus M_2$ be a silting object with $M_1$ and $M_2$
indecomposable. Assume that $M_1$ belongs to the
$\mathbb{Z}A_\infty^\infty$ component. Up to an auto-equivalence of
the form $\Phi_{P_1}^n\circ\Phi_{S_2}^{n'}$, we may assume that
$M_1= P_2$. Then, if $M_2$ belongs to the $0$-spherical component it
has to be $P_1$. Thus we assume that $M_2$ also belongs to the
$\mathbb{Z}A_\infty^\infty$ component. Then it follows from
Lemma~\ref{l:morphism-on-non-spherical-comp} that $M_2$ is
isomorphic to $\Sigma^m S_1$ for some $m\leq -1$ or to $\Sigma^m
R(1)$ for some $m\geq 0$. Observing $P_2\oplus\Sigma^m
R(1)=\Phi_{P_1}^{-m-1}\circ\Phi_{S_2}^m(P_2\oplus\Sigma^{-m-1} S_1)$
for $m\geq 0$ finishes the proof for the silting-object part.

That the simple-minded collection corresponding to a silting object is the desired one
follows from the Hom-duality they satisfy.
\end{proof}

\subsection{The silting quiver}

Recall from \cite{AiharaIyama12} that the \emph{silting quiver} has as
vertices the isomorphism classes of basic silting objects and there
is an arrow from $M$ to $M'$ if $M'$ can be obtained from $M$ by a
left mutation.

The vertex set of the silting quiver of $\cd^b(\mod\Lambda)$ is
$\{(n,n',m)\mid n,n'\in\mathbb{Z},m\in\mathbb{Z}_{\leq 0}\}$, where
$(n,n',0)$ represents the silting object
$\Phi_{P_1}^n\circ\Phi_{S_2}^{n'}(P_1\oplus P_2)$ and $(n,n',m)$
($m\leq -1$) represents the silting object
$\Phi_{P_1}^n\circ\Phi_{S_2}^{n'}(\Sigma^m S_1\oplus P_2)$. It is
straightforward to show that from each vertex $(n,n',m)$ there are precisely
two outgoing arrows whose targets are respectively
\begin{itemize}
\item[$\cdot$] $(n,n'-1,m)$ and $(n+1,n',m-1)$ if $m=0$,
\item[$\cdot$] $(n+1,n'-1,m-1)$ and $(n,n',m+1)$ if $m\leq -1$.\end{itemize}

\subsection{Hearts and the space of stability conditions}

\begin{lemma}\label{l:example-length-heart}
The heart of any $t$-structure on $\cd^b(\mod\Lambda)$ is a length
category.
\end{lemma}
\begin{proof} Let $\ca$ be the heart of a $t$-structure on
$\cd^b(\mod\Lambda)$. We will show that $\ca$ has only finitely many
isomorphism classes of indecomposable objects. Such an abelian
category must be a length category.

Due to vanishing of negative extensions, it follows from
Lemma~\ref{l:morphism-on-spherical-comp} that $\ca$ contains at most
one indecomposable object from the $0$-spherical component
respectively the $2$-spherical component.

Suppose that $\ca$ contains an indecomposable object from the
$\mathbb{Z}A_\infty^\infty$ component. Without loss of generality we
may assume that it is $P_2$. It follows from
Lemma~\ref{l:morphism-on-non-spherical-comp} that for $n\geq 3$ and
$m\in\mathbb{Z}$ either $\Hom(P_2,\Sigma^{m'}\Sigma^m R(n))\neq 0$
for some $m'<0$ or $\Hom(\Sigma^m R(n),\Sigma^{m'}P_2)\neq 0$ for
some $m'<0$. Similarly for $L(n)$. Therefore an indecomposable
object $M$ belongs to the heart only if it is isomorphic to one of
$\Sigma^m P_2$, $\Sigma^m R(1)$, $\Sigma^m R(2)$, $\Sigma^m L(1)$
and $\Sigma^m L(2)$, $m\in\mathbb{Z}$. But at most one shift of a
nonzero object can belong to a heart. So $\ca$ contains at most $7$
indecomposable objects up to isomorphism.
\end{proof}

In view of Lemma~\ref{l:example-length-heart}, the result in the
preceding subsection shows that all bounded $t$-structures on
$\cd^b(\mod\Lambda)$ are related to each other by a sequence of
left or/and right mutations. In particular, this
implies that the $t$-structures Woolf considered in~\cite[Section
3.1]{Woolf10} are already all bounded $t$-structures on
$\cd^b(\mod\Lambda)$. Therefore we have

\begin{corollary}
\begin{itemize}
\item[(a)] The Bridgeland space of stability conditions on
$\cd^b(\mod\Lambda)$ is $\mathbb{C}^2$.
\item[(b)] An abelian category is the heart of some bounded
$t$-structure on $\cd^b(\mod\Lambda)$ if and only if it is
equivalent to $\mod\Gamma$ for $\Gamma=\Lambda$ or
$\Gamma=K(\xymatrix@C=1pc{\cdot\ar[r]&\cdot})$ or $\Gamma=K\oplus K$.
\end{itemize}
\end{corollary}

%\bibliographystyle{amsplain}
%\bibliography{stanYang}
%\end{document}
\def\cprime{$'$}
\providecommand{\bysame}{\leavevmode\hbox to3em{\hrulefill}\thinspace}
\providecommand{\MR}{\relax\ifhmode\unskip\space\fi MR }
\providecommand{\MRhref}[2]{%
  \href{http://www.ams.org/mathscinet-getitem?mr=#1}{#2}
}
\providecommand{\href}[2]{#2}

\end{document}